\documentclass[a4paper,11pt]{article}

\usepackage[applemac]{inputenc}
\usepackage{enumerate}
\usepackage{lmodern}
\usepackage[T1]{fontenc}
\usepackage{verbatim}
\usepackage{textcomp}
\usepackage[english]{babel}
\usepackage[a4paper,vmargin={3.5cm,3.5cm},hmargin={2.5cm,2.5cm}]{geometry}
\usepackage[font=sf, labelfont={sf,bf}, margin=1cm]{caption}
\usepackage[pagebackref=false,
pdftex]{hyperref}
\usepackage[pdftex]{color,graphicx}

\usepackage{amsmath,amsfonts,amssymb,amsthm,mathrsfs}
\usepackage{colonequals} 
\usepackage{stmaryrd}     

\usepackage{calc}
\usepackage{graphicx}
\makeatletter
\newlength{\temp@wc@width}
\newlength{\temp@wc@height}
\newcommand{\widecheck}[1]{%
  \setlength{\temp@wc@width}{\widthof{$#1$}}%
  \setlength{\temp@wc@height}{\heightof{$#1$}}%
  #1\hspace{-\temp@wc@width}%
  \raisebox{\temp@wc@height+1pt}[\heightof{$\widehat{#1}$}]%
     {\rotatebox[origin=c]{180}{\vbox to 0pt{\hbox{$\widehat{\hphantom{#1}}$}}}}%
}
\makeatother

\usepackage{cleveref}
  \crefname{theorem}{Theorem}{Theorems}
  \crefname{lemma}{Lemma}{Lemmas}
  \crefname{remark}{Remark}{Remarks}
  \crefname{proposition}{Proposition}{Propositions}
  \crefname{definition}{Definition}{Definitions}
  \crefname{corollary}{Corollary}{Corollaries}
  \crefname{section}{Section}{Sections}
  \crefname{figure}{Figure}{Figures}

\usepackage{color}

\newtheorem{theorem}{Theorem}[]

\newtheorem{proposition}[theorem]{Proposition}
\newtheorem{lemma}[theorem]{Lemma}
\newtheorem{corollary}[theorem]{Corollary}
\theoremstyle{definition}
\newtheorem*{remark}{Remark}

\def\v{{\mathcal V}}
\def\ve{\varepsilon}
\def\wt{\widetilde}

\def\bd{{\bf d}}
\def\la{\longrightarrow}

\def\R{{\mathbb R}}
\def\P{{\mathbb P}}
\def\E{{\mathbb E}}
\def\N{{\mathbb N}}

\def\build#1_#2^#3{\mathrel{\mathop{\kern 0pt#1}\limits_{#2}^{#3}}}

\title{Typical behavior of the harmonic measure in critical Galton--Watson trees}

\author{Shen LIN 
\thanks{Supported in part by the grant ANR-14-CE25-0014 (ANR GRAAL)} \\
\small \it Ecole Normale Sup\'erieure, Paris, France \\
\small \textit{E-mail}: \texttt{shen.lin.math@gmail.com}
}

\date{}

\begin{document}

\maketitle

\begin{abstract}
We study the typical behavior of the harmonic measure of balls in large critical Galton--Watson trees whose offspring distribution has finite variance. The harmonic measure considered here refers to the hitting distribution of height $n$ by simple random walk on a critical Galton--Watson tree conditioned to have height greater than~$n$. We prove that, with high probability, the mass of the harmonic measure carried by a random vertex uniformly chosen from height $n$ is approximately equal to $n^{-\lambda}$, where the constant $\lambda >1$ does not depend on the offspring distribution. This universal constant $\lambda$ is equal to the first moment of the asymptotic distribution of the conductance of size-biased Galton--Watson trees minus 1.


\end{abstract}

\section{Introduction}
Let $\theta$ be a non-degenerate probability measure on $\mathbb{Z}_{+}$, and assume that $\theta$ has mean one and finite variance $\sigma^2>0$. Under the probability measure $\P$, for every integer $n\geq 0$, we let $\mathsf{T}^{(n)}$ be a Galton--Watson tree with offspring distribution $\theta$, conditioned on non-extinction at generation~$n$. From our assumption that $\theta$ is critical (i.e.~of mean 1), $\mathsf{T}^{(n)}$ is a.s.~a finite tree. We denote by $\mathsf{T}^{(n)}_n$ the set of all vertices of $\mathsf{T}^{(n)}$ at generation $n$. The classical Yaglom's theorem in the theory of branching processes states that $n^{-1}\# \mathsf{T}^{(n)}_n$ converges in distribution to an exponential distribution with parameter $2/\sigma^2$ (see e.g.~Theorem 9.2 in Chapter 1 of \cite{AN}).

Conditionally on the tree $\mathsf{T}^{(n)}$, we consider a simple random walk on $\mathsf{T}^{(n)}$ starting from the root. Let $\Sigma_n$ be the first hitting point of generation $n$ by the random walk. We call the distribution of $\Sigma_{n}$ the harmonic measure~$\mu_n$ at level $n$, which is a random probability measure supported on the level set $\mathsf{T}^{(n)}_n$. The main result of the present work is the following: 

\begin{theorem}\label{thm:main-dis}
Let $\Omega_n$ be a random vertex uniformly chosen from $\mathsf{T}^{(n)}_n$. There exists a universal constant $\lambda\in (1,\infty)$, which does not depend on the offspring distribution $\theta$, such that for every $\delta>0$,
\begin{eqnarray}\label{eq:thm-main}
\lim_{n\to \infty}\P \Big( n^{-\lambda-\delta} \leq \mu_n(\Omega_n) \leq n^{-\lambda+\delta} \Big)= 1.
\end{eqnarray}
\end{theorem}

Loosely speaking, if we look at a typical vertex at level $n$ of the conditional critical Galton--Watson tree $\mathsf{T}^{(n)}$, then as shown by (\ref{eq:thm-main}), it carries with high probability a mass of order $n^{-\lambda}$ given by the harmonic measure $\mu_n$. Recall that the number of vertices of $\mathsf{T}^{(n)}$ at generation $n$ is of order $n$ according to Yaglom's theorem. Since $\lambda>1$, our main theorem clearly indicates that the harmonic measure $\mu_n$ is far from being uniformly spread over $\mathsf{T}^{(n)}_n$. 

To be more precise, by the definition of $\Omega_n$, the convergence (\ref{eq:thm-main}) can be rewritten as
\begin{equation*}
\lim_{n\to \infty} \E\Bigg[ \frac{1}{\#\mathsf{T}^{(n)}_n} \sum\limits_{v \in \mathsf{T}^{(n)}_n} \mathbf{1}\{ \mu_n(v) >n^{-\lambda+\delta} \mbox { or }  \mu_n(v) <n^{-\lambda-\delta} \} \Bigg] =0.
\end{equation*}
Using the theorem of Yaglom, it is easy to see that the preceding convergence is equivalent to
\begin{equation}
\label{eq:11}
\lim_{n\to \infty} \frac{1}{n}\,\E\Bigg[ \sum\limits_{v \in \mathsf{T}^{(n)}_n} \mathbf{1}\{ \mu_n(v) >n^{-\lambda+\delta} \mbox { or }  \mu_n(v) <n^{-\lambda-\delta} \}\Bigg] =0.
\end{equation}
We take $\delta\in (0,\lambda-1)$, and define $A_n\colonequals \{v\in \mathsf{T}^{(n)}_n\colon \mu_n(v)> n^{-\lambda+\delta}\}$. The convergence (\ref{eq:11}) implies that for every $\varepsilon>0$,
\begin{displaymath}
\P\Big(\frac{\# A_n}{n}>\varepsilon\Big) \build{\la}_{n\to \infty}^{} 0.
\end{displaymath}
On the other hand, by the definition of $A_n$, we have $\mu_n(v)\leq n^{-\lambda+\delta}$ for any vertex $v\in \mathsf{T}^{(n)}_n\backslash A_n$, and it follows that $\mu_n(\mathsf{T}^{(n)}_n\backslash A_n)\leq n^{-\lambda+\delta}\# \mathsf{T}^{(n)}_n$. Using again Yaglom's theorem, we get that
\begin{displaymath}
\P\Big(\mu_n(\mathsf{T}^{(n)}_n\backslash A_n)>\varepsilon \Big) \build{\la}_{n\to \infty}^{} 0.
\end{displaymath}
Therefore, it holds with probability tending to 1 as $n\to \infty$ that, up to a mass of size $\varepsilon$, the harmonic measure $\mu_n$ is supported on a subset of $\mathsf{T}^{(n)}_n$ of cardinality smaller than $\varepsilon n$. This simple consequence of Theorem~\ref{thm:main-dis} has already been observed in a recent paper of Curien and Le Gall~\cite{CLG13}, where, instead of looking at a vertex typical for the tree, they have considered a vertex typical for the harmonic measure. In Theorem~1 of~\cite{CLG13}, they have shown the existence of another universal constant $\beta\in (0,1)$ independent of the offspring distribution~$\theta$, such that for every $\delta>0$, we have the convergence in $\P$-probability
\begin{equation}
\label{eq:theorem1-CLG}
 \mu_n\big( \{v\in \mathsf{T}_n^{(n)} \colon n^{-\beta-\delta} \leq \mu_n(v) \leq n^{-\beta+\delta}\} \big) \xrightarrow[n\to\infty]{(\P)}  1.
\end{equation}
In other words, the typical mass given by the harmonic measure $\mu_n$ to a vertex at level $n$ drawn with respect to the harmonic measure is with high probability of order $n^{-\beta}$ with $\beta<1$, and we can thus say that most of the harmonic measure $\mu_n$ is supported on a subset of size approximately equal to $n^\beta$, which is much smaller than the size of $\mathsf{T}^{(n)}_n$. One may think of both results, \eqref{eq:thm-main} and \eqref{eq:theorem1-CLG}, as the first steps towards a complete analysis of the multifractal nature of the harmonic measure in critical Galton--Watson trees.

As pointed out by Curien and Le Gall in~\cite{CLG13}, for studying the harmonic measure at level $n$ on $\mathsf{T}^{(n)}_n$, we can consider directly simple random walk on the reduced tree $\mathsf{T}^{*n}$, which consists of all vertices of $\mathsf{T}^{(n)}$ that have at least one descendant at generation $n$. Moreover, if we scale the graph distances by $n^{-1}$, the rescaled discrete reduced trees $n^{-1}\mathsf{T}^{*n}$ converge to a random compact rooted $\R$-tree $\Delta$, whose structure is described as follows. We take a random variable $U_{\varnothing}$ uniformly distributed over $[0,1]$, and start with an oriented line segment of length $U_{\varnothing}$, whose origin will be the root of $\Delta$. To the other end of this initial line segment, we attach the initial points of two new line segments with respective lengths $U_1$ and $U_2$, in such a way that, conditionally given $U_{\varnothing}$, the variables $U_1$ and $U_2$ are independent and uniformly distributed over $[0,1-U_{\varnothing}]$. To the other end of the first of these 2 line segments, we attach two line segments whose lengths are independent and uniformly distributed over $[0,1-U_{\varnothing}-U_1]$, again conditionally on $U_{\varnothing}$ and~$U_1$. We repeat this procedure independently for the second line segment with $U_1$ replaced by~$U_2$. We continue this construction by induction, and after an infinite number of steps we obtain a random non-compact rooted $\R$-tree $\Delta_0$ with binary branching (see Figure~\ref{figure:delta2}), whose completion with respect to its intrinsic metric $\mathbf{d}$ is called the continuous reduced tree $\Delta$. We assume that $\Delta$ is also defined under the probability measure $\P$. Its boundary $\partial \Delta$ is the set of all points of $\Delta$ at height 1, i.e.~at distance 1 from the root. 

Brownian motion on $\Delta$ starting from the root can be easily defined up to the first hitting time of $\partial \Delta$. Roughly speaking, this process behaves in the same way as a standard linear Brownian motion as long as it remains inside a line segment. It is reflected at the root of the tree $\Delta$, and when it arrives at a branching point, it chooses each of the three line segments incident to this point with equal probabilities. The (continuous) harmonic measure $\mu$ on $\partial \Delta$ is defined as the (quenched) distribution of the first hitting point of $\partial \Delta$ by Brownian motion.

\begin{figure}[!h]
\begin{center}
\includegraphics[width=13cm]{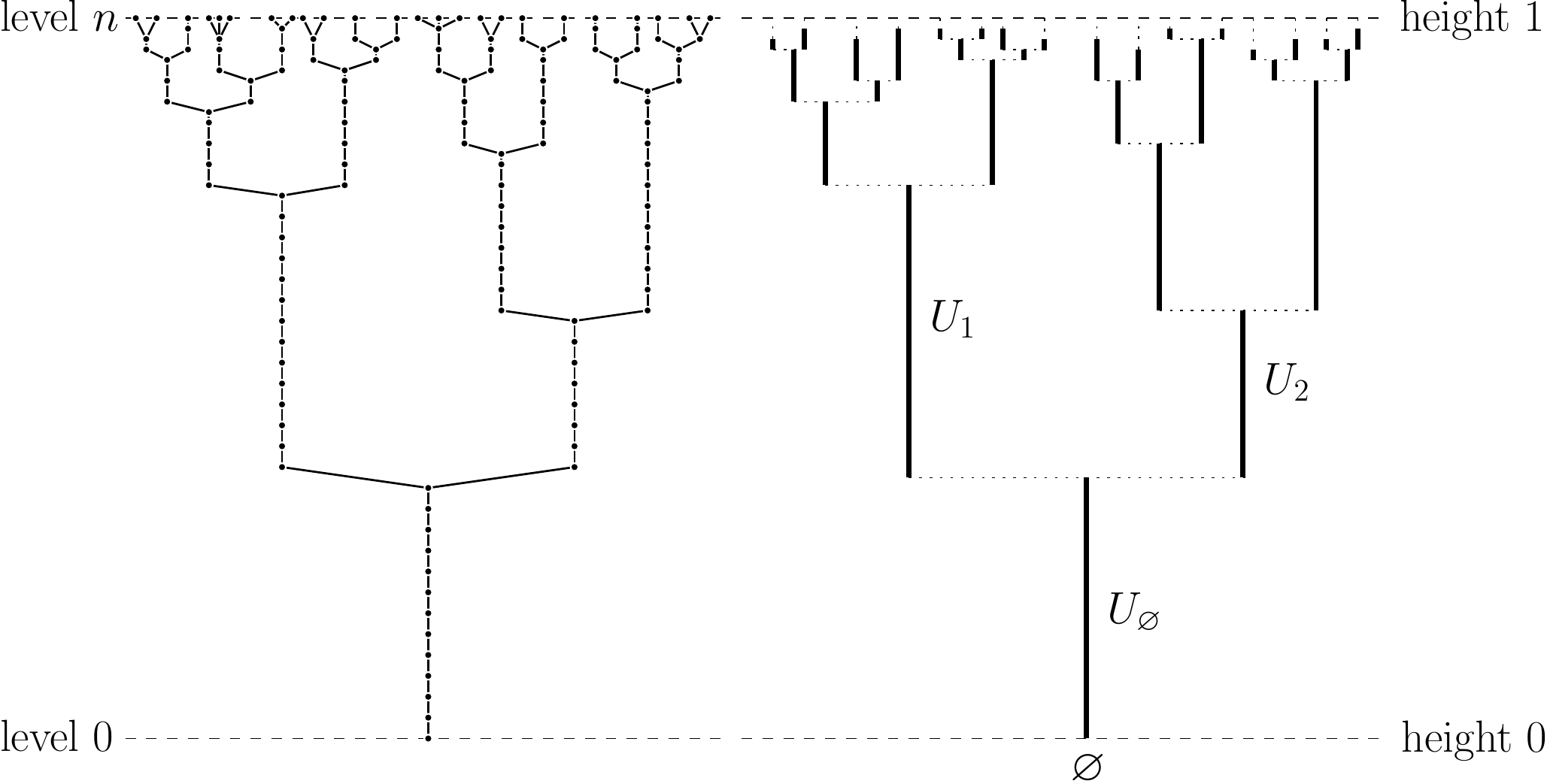}
\caption{A large reduced tree $\mathsf{T}^{*n}$ of height $n$ on the left, and the random tree $\Delta_0$ on the right\label{figure:delta2}}
\end{center}
\end{figure}

We then define another (non-compact) random rooted $\R$-tree $\Gamma$ with binary branching, such that each point of $\Gamma$ at height $y\in [0,\infty)$ corresponds to a point of $\Delta_0$ at height $1-e^{-y}\in [0,1)$. The resulting new tree $\Gamma$ is the Yule tree which describes the genealogy of the classical Yule process, where individuals have independent exponential lifetimes with parameter 1 and each individual has exactly two offspring. By definition, the boundary $\partial \Gamma$ of $\Gamma$ is the set of all infinite geodesics in $\Gamma$ starting from the root (these are called geodesic rays). Due to the binary branching mechanism, both $\partial \Delta$ and $\partial \Gamma$ can be canonically identified with $\{1,2\}^{\N}$.

For every $r>0$, we write $\Gamma_r$ for the level set of $\Gamma$ at height $r$. By a martingale argument, one can define
\begin{displaymath}
\mathcal{W}\colonequals \lim_{r\to \infty} e^{-r}\# \Gamma_r,
\end{displaymath}
and it is well known that $\mathcal{W}$ follows an exponential distribution of parameter 1. For every $x\in \Gamma$, we let $H(x)$ denote the height of $x$ in $\Gamma$, and we write $\Gamma[x]$ for the tree of descendants of $x$ in $\Gamma$, viewed as an infinite random $\R$-tree rooted at $x$. For every $r>0$, we write $\Gamma_r[x]$ for the level set at height $r$ of the tree $\Gamma[x]$. If one thinks of $\Gamma[x]$ as a subtree of $\Gamma$, the set $\Gamma_r[x]$ consists of all the points of $\Gamma$ at height $r+H(x)$ that are descendants of $x$. We similarly define
\begin{displaymath}
\mathcal{W}_{x}\colonequals \lim_{r\to \infty} e^{-r}\# \Gamma_r[x].
\end{displaymath}
It is immediate to see that for every $r>0$,
\begin{displaymath}
\sum\limits_{x\in \Gamma_r} e^{-r}\mathcal{W}_x= \mathcal{W}.
\end{displaymath}
The uniform measure $\bar \omega$ on $\partial \Gamma$ is defined as the unique random probability measure on $\partial \Gamma$ satisfying that, for every $x\in \Gamma$ and for every geodesic ray $\mathbf{v}\in \partial \Gamma$ passing through $x$,
\begin{displaymath}
\bar \omega(\mathcal{B}(\mathbf{v},H(x)))=  e^{-H(x)}\frac{\mathcal{W}_x}{\mathcal{W}},
\end{displaymath}
where $\mathcal{B}(\mathbf{v},H(x))$ stands for the set of all geodesic rays in $\Gamma$ that coincide with $\mathbf{v}$ up to height $H(x)$. In earlier work, $\bar \omega$ is also named as the branching measure on the boundary of $\Gamma$. Recall that $\partial \Delta$ can be identified with $\partial \Gamma$ as explained above. We let $\omega$ be the random probability measure on $\partial \Delta$ induced by $\bar \omega$, which will be referred to as the uniform measure on $\partial \Delta$.

\begin{theorem}\label{thm:main-continu}
With the same constant $\lambda$ as in Theorem~\ref{thm:main-dis}, we have $\P$-a.s.~$\omega(\mathrm{d}\mathbf{v})$-a.e.
\begin{eqnarray}
\lim_{r\downarrow 0} \frac{\log \mu (\mathcal{B}_{\bd}(\mathbf{v},r))}{\log r} &= &\lambda \,, \label{eq:loc-dim-harm}\\
\lim_{r\downarrow 0} \frac{\log \omega (\mathcal{B}_{\bd}(\mathbf{v},r))}{\log r} &=  &1 \,,\label{eq:loc-dim-unif}
\end{eqnarray}
where $\mathcal{B}_{\bd}(\mathbf{v},r)$ stands for the closed ball of radius $r$ centered at $\mathbf{v}$ in the metric space $(\Delta,\bd)$.
\end{theorem}

\remark The Hausdorff measure of $\partial \Delta$ with respect to $\mathbf{d}$ is a.s.~equal to 1. An exact Hausdorff measure function can be found in Duquesne and Le Gall~\cite[Theorem 1.3]{DLG06}.

\begin{corollary}
\label{corol-intro}
$\P$-a.s.~the two measures $\mu$ and $\omega$ on the boundary of $\Delta$ are mutually singular.
\end{corollary}

In fact, with the same constant $\beta$ as in~(\ref{eq:theorem1-CLG}), it is shown in Theorem 3 of~\cite{CLG13} that $\P$-a.s.~$\mu(\mathrm{d}\mathbf{v})$-a.e.,
\begin{displaymath}
\lim_{r\downarrow 0} \frac{\log \mu (\mathcal{B}_{\bd}(\mathbf{v},r))}{\log r} = \beta.
\end{displaymath}
If we define
\begin{displaymath}
B=\big\{\mathbf{v}\in \partial \Delta \colon \lim_{r\downarrow 0} \frac{\log \mu (\mathcal{B}_{\bd}(\mathbf{v},r))}{\log r} = \beta \big\},
\end{displaymath}
then $\P$-a.s.~$\mu(B)=1$. However, since $\beta<1<\lambda$, we get from~(\ref{eq:loc-dim-harm}) that $\P$-a.s.~$\omega(B)=0$, which finishes the proof of Corollary \ref{corol-intro}.

Similar results for supercritical infinite Galton--Watson trees can be found in Theorem~3 of Liu and Rouault~\cite{LR97} and in Theorem 6.3 of Lyons, Pemantle and Peres~\cite{LPP95}, where the uniform measure and the visibility measure (defined as the law of the geodesic ray chosen by \emph{forward} simple random walk) on the boundary of the infinite tree are considered.

In order to get a better understanding of the distinguished geodesic ray in the Yule tree~$\Gamma$ chosen randomly according to the uniform measure $\bar \omega(\mathrm{d}\mathbf{v})$, we follow the ideas of Lyons, Pemantle and Peres~\cite{LPP95b} to construct a size-biased version $\widehat \Gamma$ of $\Gamma$, which is the genealogical tree of the following branching process. Initially, there is one particle having an exponential lifetime with parameter~2, and it reproduces two offspring simultaneously when it dies. We choose one of them uniformly at random and the chosen one will continue as the initial ancestor, while the other offspring will independently evolve as the classical Yule process. The size-biased Yule tree $\widehat \Gamma$ thus defined is an infinite random $\R$-tree with binary branching. Applying to $\widehat \Gamma$ the same transformation that relates $\Gamma$ and $\Delta_0$, we get a bounded (yet non-compact) rooted $\R$-tree $\widehat \Delta_0$, which is interpreted as the size-biased version of $\Delta_0$. Essentially, every point of $\widehat \Gamma$ at height $y\geq 0$ corresponds to a point of $\widehat \Delta_0$ at height $1-e^{-y}$. The completion of $\widehat \Delta_0$ with respect to its intrinsic metric is denoted as $\widehat \Delta$, and we call $\widehat \Delta$ the size-biased reduced tree. Its boundary $\partial \widehat\Delta$ is similarly defined as the set of all points of $\widehat \Delta$ at height 1. 

\begin{figure}[!h]
\begin{center}
\includegraphics[width=7.5cm]{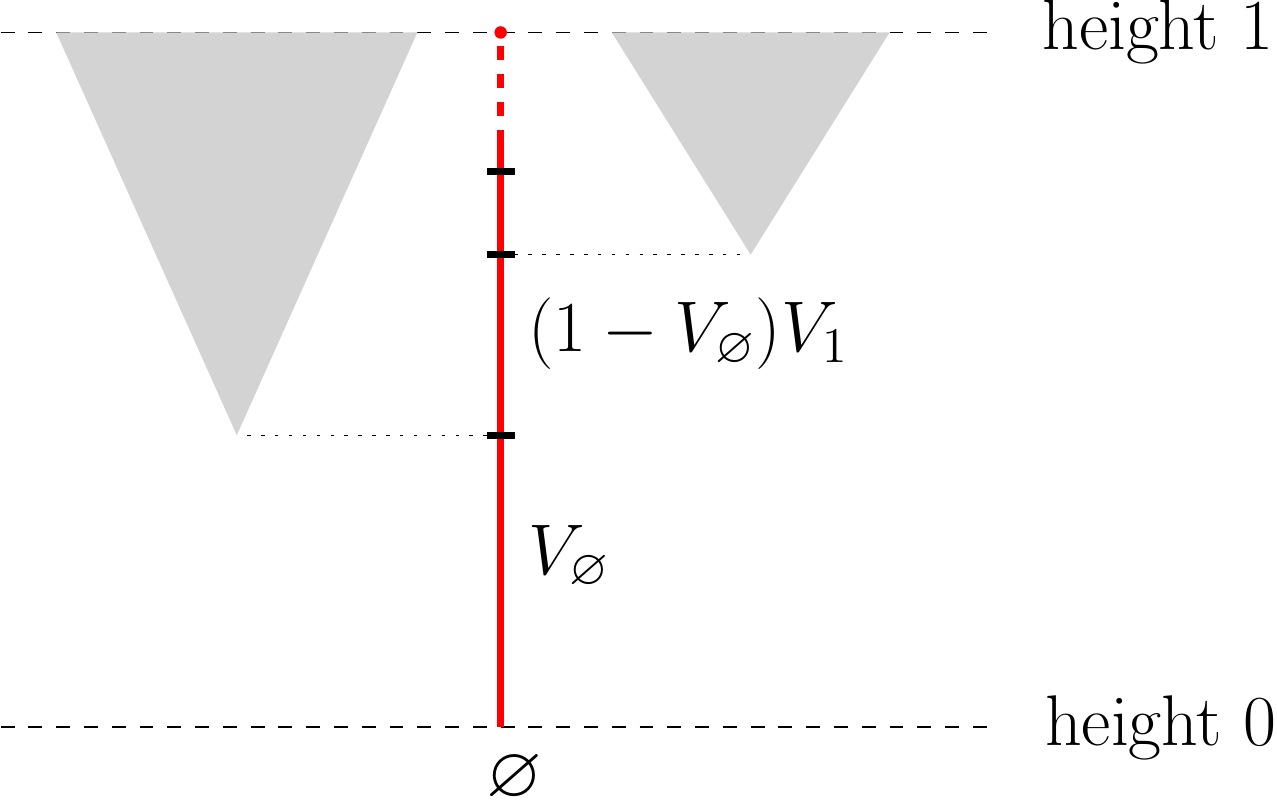}
\caption{Schematic representation of the size-biased reduced tree $\widehat \Delta$ \label{figure:bias}}
\end{center}
\end{figure}

Due to the previous description of $\widehat \Gamma$, one can also construct $\widehat \Delta$ directly as follows. At first, the root $\varnothing$ of $\widehat \Delta$ has a distinguished descendant line of length 1. Let $V_{\varnothing}$ be a random variable taking values in $[0,1]$ with density $2(1-x)$, and we graft to the distinguished descendant line at height $V_{\varnothing}$ a subtree which is an independent copy of $\Delta$ scaled by the factor $(1-V_{\varnothing})$. In the second step, we take $V_1$ as an independent copy of $V_{\varnothing}$ and graft to the distinguished descendant line at height $V_{\varnothing}+(1-V_{\varnothing})V_1$ another independent copy of $\Delta$ scaled by the factor $(1-V_{\varnothing})(1-V_1)$. Note that for each grafting, we choose the left-hand side or the right-hand side of the distinguished descendant line with equal probabilities. We continue this procedure to graft more subtrees to the distinguished descendant line, with the height of the grafting position increasing to 1. After an infinite number of steps we obtain a realization of $\widehat \Delta$. See Figure~\ref{figure:bias} for an illustration. We assume as before that $\widehat \Delta$ is defined under the probability measure~$\P$.

The constant $\lambda$ appearing in Theorems~\ref{thm:main-dis} and~\ref{thm:main-continu} can be expressed in terms of the (continuous) conductance of $\widehat \Delta$. Informally, if we think of the random trees $\Delta$ and $\widehat \Delta$ as electric networks of resistors with unit resistance per unit length, the effective conductances between the root and the boundary in $\Delta$ and $\widehat \Delta$ are continuous random variables denoted respectively as $\mathcal{C}$ and~$\widehat{\mathcal{C}}$. From a probabilistic point of view, each of these conductances can be obtained as the mass under the Brownian excursion measure in the corresponding tree for the excursion paths away from the root that hit height~1. It is easy to see that both $\mathcal{C}$ and $\widehat{\mathcal{C}}$ take values in $[1,\infty)$. The law of $\mathcal{C}$ has been studied at length in~\cite{CLG13}. Following the above construction of $\widehat \Delta$ and the electric network interpretation, the distribution of $\widehat{\mathcal{C}}$ satisfies the recursive distributional equation
\begin{equation}
\label{eq:c*-rde}
\widehat{\mathcal{C}}\,\overset{(\mathrm{d})}{=\joinrel=}\, \Big(V + \frac{1-V}{\widehat{\mathcal{C}}+ \mathcal{C}}\Big)^{-1},
\end{equation}
where in the right-hand side $V,\mathcal{C}$ and $\widehat{\mathcal{C}}$ are independent, and the distribution of the random variable $V$ has density $2(1-x)$ over $[0,1]$. Using this distributional identity \eqref{eq:c*-rde}, we prove that, similarly as the law of $\mathcal{C}$, the law $\widehat \gamma$ of $\widehat{\mathcal{C}}$ has finite moments of all orders, and it has a continuous density $\widehat f$ over $[1,\infty)$, which reaches its global maximum at $3/2$. The density function $\widehat f$ exhibits a singular behavior analogous to that of the density function of~$\mathcal{C}$ (see~\cite[Section 2.3]{CLG13}). Although $\widehat f$ is twice continuously differentiable on the interval $(1,3)$, it is shown that $\widehat f$ is not third-order differentiable at the point~2. A similar singular behavior is expected at all integer points $n\geq 2$. See Figure~\ref{figure:conductancebias}.

\begin{figure}[!h]
\begin{center}
\includegraphics[width=11cm]{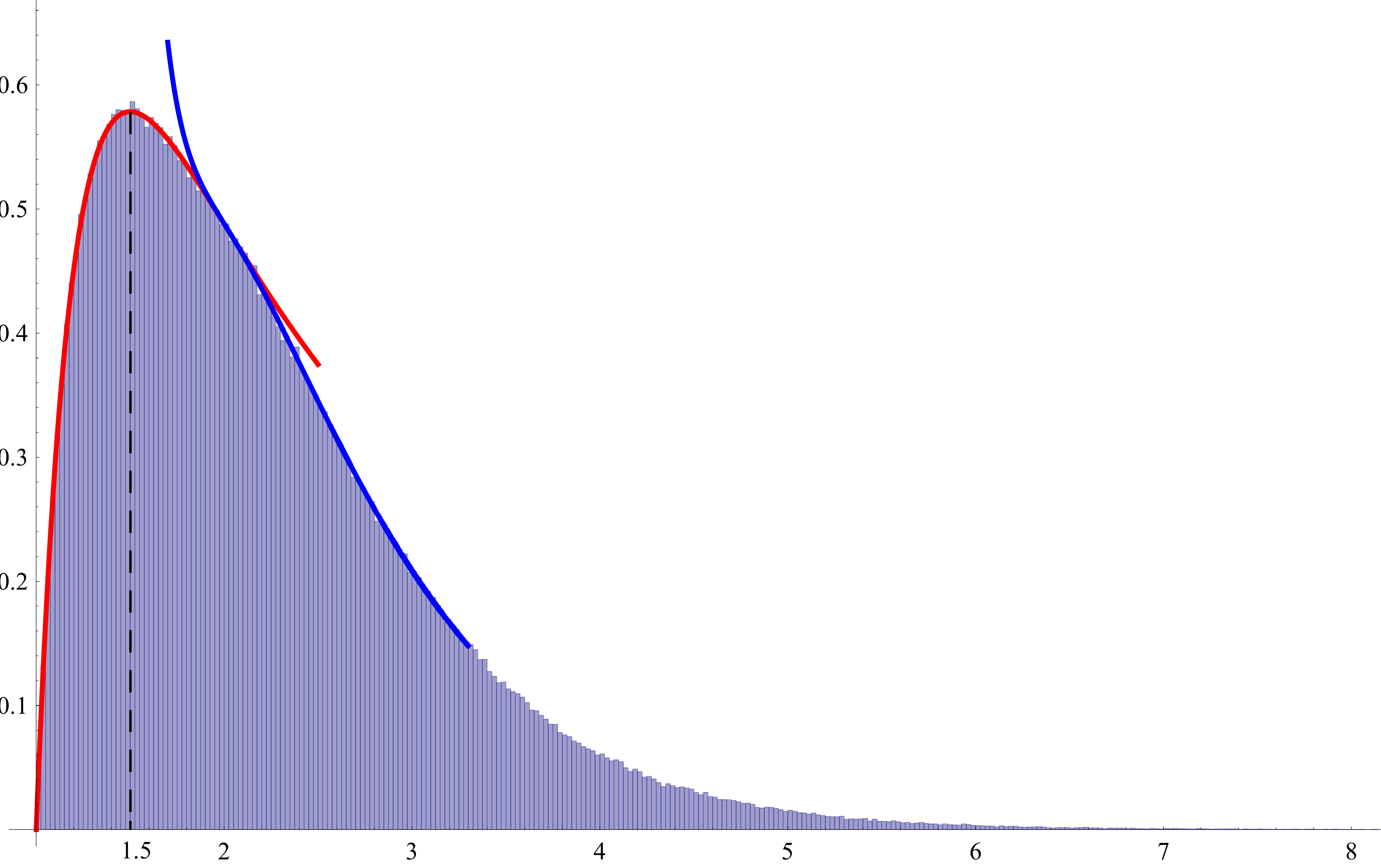}
\caption{A histogram of the distribution $\widehat \gamma$ over $(1,\infty)$ obtained from the simulations based on~(\ref{eq:c*-rde}). The red and the blue curves correspond respectively to the explicit formulae for the density of $\widehat \gamma$ over $[1,2]$ and $[2,3]$. \label{figure:conductancebias}}
\end{center}
\end{figure}

\begin{proposition}
\label{prop:lambda-expression}
The distribution $\widehat\gamma$ of the conductance $\widehat{\mathcal{C}}$ is characterized in the class of all probability measures on $[1,\infty)$ by the distributional equation~(\ref{eq:c*-rde}). The constant $\lambda$ appearing in Theorems~\ref{thm:main-dis} and~\ref{thm:main-continu} is given by
\begin{equation}
\label{eq:lambda-expression}
\lambda = \E\big[\,\widehat{\mathcal{C}}\,\big]-1 \in (1,\infty).
\end{equation}
\end{proposition}

Numerical simulations based on \eqref{eq:c*-rde} and \eqref{eq:lambda-expression} show that $\lambda \approx 1.21$.

It is worth pointing out that the main results of \cite{CLG13} have been generalized in \cite{LIN} to the case where the critical offspring distribution $\theta$ belongs to the domain of attraction of a stable distribution of index $\alpha\in (1,2]$. We expect that results analogous to those in the present work should hold in the general stable case. This is left to be explored in a future work. 

The remainder of this paper is structured as follows. We start by defining formally the continuous random trees $\Delta$ and $\Gamma$. The notation of the random variables involved will be slightly different from the one used in this Introduction. The distribution $\widehat\gamma$ of the conductance $\widehat{\mathcal{C}}$ is studied in Section~\ref{sec:continu-conductance}, and the proof of Theorem~\ref{thm:main-continu} and of formula~(\ref{eq:lambda-expression}) is given in Section~\ref{sec:proof-continu}. The size-biased continuous random trees $\widehat\Delta$ and $\widehat\Gamma$ are properly defined respectively in Section~\ref{sec:size-biased-yule} and in Section~\ref{sec:size-biased-reduced}. In Section~\ref{sec:discrete-setting-typical}, we gather some preliminaries for the discrete setting, where, for example, we prove the convergence of discrete conductances, and introduce a backward version of the discrete size-biased Galton-Watson tree. The last part, Section~\ref{sec:proof-thm1}, is devoted to the proof of Theorem~\ref{thm:main-dis}.

\smallskip
\noindent \textbf{Acknowledgments.} The author gratefully acknowledges many helpful suggestions of J.-F.~Le~Gall during the preparation of this work. He is also indebted to N.~Curien for several stimulating discussions.

\section{The continuous setting}

\subsection{The continuous reduced tree $\Delta$}
\label{sec:reduced-tree-delta}
We set
$$\mathcal{V} \colonequals  \bigcup_{n=0}^\infty \{1,2\}^n,$$
where $\{1,2\}^0=\{\varnothing\}$. If $v=(v_1,\ldots,v_n)\in\mathcal{V}$, we set $|v|=n$ (in particular $|\varnothing|=0$). If $v \neq \varnothing$, we define the parent of $v$ as $\bar v=(v_1,\ldots,v_{n-1})$, and we then say that $v$ is a child of~$\bar v$. If both $u=(u_1,\ldots,u_m)$ and $v=(v_1,\ldots,v_n)$ belong to $\mathcal{V}$, their concatenation is $uv\colonequals (u_1,\ldots,u_m,v_1,\ldots,v_n)$. The notions of a descendant and an ancestor of an element of $\v$ are defined in the obvious way, with the convention that a vertex $v\in \v$ is both an ancestor and a descendant of itself. If $v,w\in\v$, $v\wedge w$ is the unique element of $\v$ that is an ancestor of both $v$ and $w$ and such that $|v\wedge w|$ is maximal.

We consider a collection $(U_v)_{v\in\mathcal{V}}$ of independent real random variables uniformly distributed over $[0,1]$ under the probability measure $\P$. We set $Y_\varnothing=U_\varnothing$ and then by induction, for every $v\in \{1,2\}^n$ with $n\geq 1$, we let 
$$Y_v = Y_{\bar v} + U_v(1- Y_{\bar v}).$$
Note that a.s., $0\leq Y_v< 1$ for every $v\in\mathcal{V}$. Consider the set
$$\Delta_0 \colonequals (\{\varnothing\}\times [0,Y_\varnothing] )\cup  \Bigg(\bigcup_{v\in\mathcal{V}\backslash\{\varnothing\}} \{v\} \times (Y_{\bar v}, Y_v]\Bigg).$$
We can define a natural metric $\bd$ on $\Delta_0$, so that $(\Delta_0,\bd)$ is a (noncompact) $\R$-tree and, for every $x=(v,r)\in \Delta_0$, $\bd((\varnothing,0), x)=r$. To be specific, let $x=(v,r)\in\Delta_0$ and $y=(w,r')\in\Delta_0$:
\begin{enumerate}
\item[$\bullet$] If $v$ is a descendant of $w$ or $w$ is a descendant of $v$, we set $\bd(x,y)= |r-r'|$.
\item[$\bullet$] Otherwise, $\bd(x,y)= \bd((v\wedge w,Y_{v\wedge w}),x)+ \bd((v\wedge w,Y_{v\wedge w}),y)
=(r-Y_{v\wedge w})+(r'-Y_{v\wedge w})$.
\end{enumerate}
See Figure~\ref{figure:reduced} for an illustration of the precompact tree $\Delta_0$.

\begin{figure}[!h]
\begin{center}
\includegraphics[width=11cm]{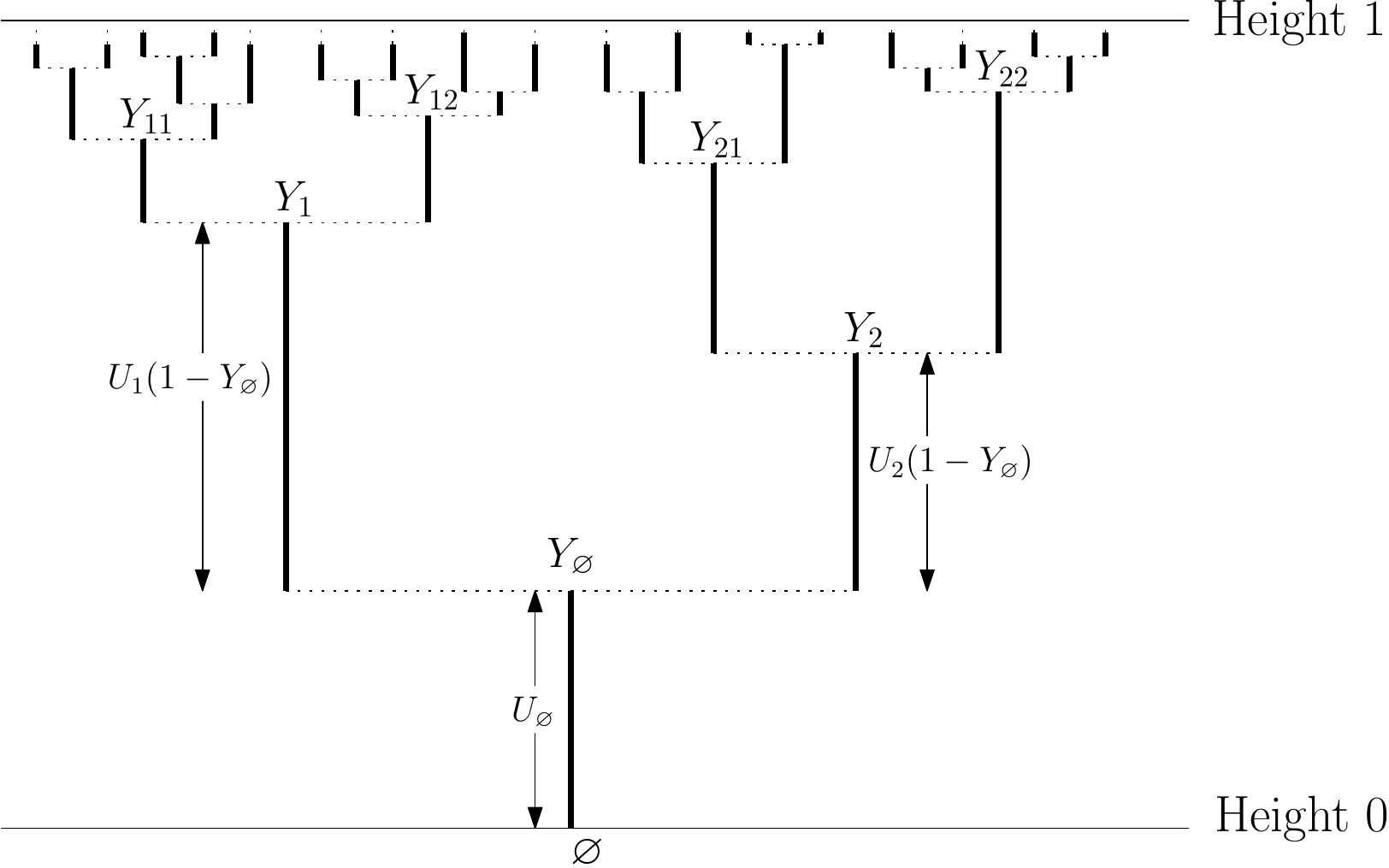}
\caption{The random tree $\Delta_0$ \label{figure:reduced}}
\end{center}
\end{figure}

We let $\Delta$ be the completion of $\Delta_0$ with respect to the metric $\bd$. Then $\Delta=\Delta_0 \cup \partial \Delta$, where the boundary $\partial \Delta \colonequals \{x\in\Delta \colon \bd((\varnothing,0), x)=1\}$ is canonically identified with $\{1,2\}^\N$. Note that $(\Delta,\bd)$ is now a compact $\R$-tree. The point $(\varnothing,0)$ is called the root of $\Delta$. For every $x\in\Delta$, we set $H(x)=\bd((\varnothing,0), x)$ and call it the height of $x$. We can define a genealogical order on $\Delta$ by setting $x\prec y$ if and only if $x$ belongs to the geodesic path from the root to $y$.

For every $\ve\in(0,1)$, we set $\Delta_\ve \colonequals \{x\in\Delta\colon H(x)\leq 1-\ve\}$, which is also a compact $\R$-tree for the metric $\bd$. The leaves of $\Delta_\ve$ are the points of the form $(v,1-\ve)$ for all $v\in\mathcal{V}$ such that $Y_{\bar v}< 1-\ve\leq Y_v$. The branching points of $\Delta_\ve$ are the points of the form $(v,Y_v)$ for all $v\in \mathcal{V}$ such that $Y_v<1-\ve$.

Conditionally on $\Delta$, we can take any $\varepsilon\in (0,1)$ and define Brownian motion on $\Delta_{\ve}$ starting from the root. Informally, this process behaves like linear Brownian motion as long as it stays on an ``open interval'' of the form $\{v\}\times (Y_{\bar v},Y_v\wedge (1-\ve))$, and it is reflected at the root $(\varnothing,0)$ and at the leaves of $\Delta_{\ve}$. When it arrives at a branching point of the tree, it chooses each of three possible line segments ending at this point with equal probabilities. By taking a sequence $(\ve_n=2^{-n})_{n\geq 1}$ and then letting $n$ go to infinity, we can construct under the same probability measure $P$ a Brownian motion $(B_t)_{t\geq 0}$ on $\Delta$ starting from the root up to its first hitting time $\tau$ of $\partial \Delta$. We refer the reader to~\cite[Section 2.1]{CLG13} for the details of this construction. The harmonic measure $\mu$ is the distribution of $B_{\tau-}$ under~$P$, which is a (random) probability measure on $\partial \Delta=\{1,2\}^{\N}$.

\subsection{The Yule tree $\Gamma$}
\label{sec:yule-tree}
To define the Yule tree, consider a collection $(I_v)_{v\in\mathcal{V}}$ of independent real random variables exponentially distributed with mean $1$ under the probability measure $\P$. We set $Z_\varnothing=I_\varnothing$
and then by induction, for every $v\in \{1,2\}^n$ with $n\geq 1$, $Z_v = Z_{\bar v} + I_v$. The Yule tree is the set
$$\Gamma \colonequals (\{\varnothing\}\times [0,Z_\varnothing]) \cup  \Bigg(\bigcup_{v\in\mathcal{V}\backslash\{\varnothing\}} \{v\} \times (Z_{\bar v}, Z_v]\Bigg),$$
which is equipped with the metric $d$ defined in the same way as $\bd$ in the preceding section. For every $x=(v,r)\in\Gamma$, $d((\varnothing,0),x)=r$ and we keep the notation $H(x)=r$ for the height of the point $x$.

Observe that if $U$ is uniformly distributed over $[0,1]$, the random variable $-\log(1-U)$ is exponentially distributed with mean $1$. Hence, we may and will suppose that the collection $(I_v)_{v\in\mathcal{V}}$ is constructed from $(U_v)_{v\in\mathcal{V}}$ in the previous section via the formula $I_v=-\log(1-U_v)$, for every $v\in\mathcal{V}$. Then the mapping $\Psi$ defined on $\Delta_0$ by $\Psi(v,r)=(v,-\log(1-r))$, for every $(v,r)\in\Delta_0$, is a homeomorphism from $\Delta_0$ onto $\Gamma$.

Using stochastic calculus, we can write, for every $t\in[0,\tau)$,
\begin{equation}
\label{BM-Yule-reduced}
\Psi(B_t) =W\Big(\int_0^t (1-H(B_s))^{-2}\,\mathrm{d}s\Big)
\end{equation}
where $(W(t))_{t\geq 0}$ is Brownian motion with constant drift $1/2$ towards infinity on the Yule tree (this process is defined in a similar way as Brownian motion on $\Delta$, except that it behaves like Brownian motion with drift $1/2$ on every ``open interval'' of the tree $\Gamma$). Note that $W$ is also defined under the probability measure $P$. From now on, Brownian motion on the Yule tree $\Gamma$ or on other similar trees will always refer to Brownian motion with drift $1/2$ towards infinity.

By definition, the boundary $\partial \Gamma$ is the set of all geodesic rays in $\Gamma$ starting from the root $(\varnothing,0)$. From the transience of Brownian motion on $\Gamma$, there is a.s.~a unique geodesic ray denoted by $W_\infty$ that is visited by $(W(t),t\geq 0)$ at arbitrarily large times. The distribution of the exit ray $W_\infty$ under $P$ yields a probability measure $\nu$ on the boundary $\partial \Gamma$. Thanks to \eqref{BM-Yule-reduced}, we have in fact $\nu=\mu$, provided we identify $\partial \Delta$ and $\partial \Gamma$ with $\{1,2\}^\mathbb{N}$ and view both $\mu$ and $\nu$ as (random) probability measures on $\{1,2\}^\mathbb{N}$.

\medskip
\noindent{\bf Yule-type trees.}
We define $\mathscr{T}$ to be the set of all collections $(z_v)_{v\in\v}$
of positive real numbers such that the following properties hold:
\begin{enumerate}
\item[\rm(i)] $z_{\bar v}< z_v$ for every $v\in\v\backslash\{\varnothing\}$;
\item[\rm(ii)] for every $\mathbf{v}=(v_1,v_2,\ldots)\in \{1,2\}^\N$, $z_{(v_1,\ldots,v_n)} \rightarrow +\infty$ as $n\to \infty$.
\end{enumerate}
Notice that we allow the possibility that $z_\varnothing=0$.
If $(z_v)_{v\in\v}\in\mathscr{T}$, we consider the associated ``tree''
$$T\colonequals (\{\varnothing\}\times [0,z_\varnothing]) \cup  \Bigg(\bigcup_{v\in\mathcal{V}\backslash\{\varnothing\}} \{v\} \times (z_{\bar v}, z_v]\Bigg),$$
which is equipped with the distance $d$ similarly defined as above. If $x=(v,r)\in T$, we still write $H(x)=r$ for the height of $x$. The genealogical (partial) order on $T$ is defined as previously and will again be denoted by $\prec$. The set of all geodesic rays in $T$ is called the boundary $\partial T$, which is naturally identified with $\{1,2\}^\N$. If $\mathbf{u}=(u_1,u_2,\ldots)\in\{1,2\}^\N$, and $x=(v,r)\in T$, we write $x\prec \mathbf{u}$ if $v=(u_1,u_2,\ldots,u_k)$ for some integer $k\geq 0$.

We will often say that we consider a tree $T\in \mathscr{T}$: this means that we are given a collection $(z_v)_{v\in\v}$ satisfying the above properties, and we consider the associated tree $T$. In particular, the tree $T$ has an order structure given by the lexicographical order on $\v$. Elements of $\mathscr{T}$ will be called Yule-type trees. The Yule tree $\Gamma$ can be viewed as a random variable taking values in $\mathscr{T}$, and we write $\Theta(\mathrm{d}T)$ for its distribution.

Let us fix $T\in\mathscr{T}$. If $r>0$, the level set at height $r$ is $T_r \colonequals \{x\in T\colon H(x)=r\}$. If $x\in T_r$, we can then consider the subtree $T[x]$ of descendants of $x$ in $T$. Formally, we view $T[x]$ as an element of $\mathscr{T}$: We write $v_x$ for the unique element of $\v$ such that $x=(v_x,r)$, and define $T[x]$ as the Yule-type tree corresponding to the collection $(z_{v_xv}-r)_{v\in\v}$.

As we have seen in the Introduction, the limit $\mathcal{W}(T)=\lim_{r\to \infty} e^{-r}\# T_r$ exists $\Theta(\mathrm{d}T)$-a.s., and $\int \mathcal{W}(T) \Theta(\mathrm{d}T)=1$. For every $x\in T$, we similarly set $\mathcal{W}(T[x])=\lim_{r\to \infty} e^{-r}\# T_r[x]$. If $\mathbf{v}\in \partial T$ is a geodesic ray passing through $x$, let $\mathcal{B}(\mathbf{v},H(x))$ denote the set of geodesic rays in $T$ that coincide with $\mathbf{v}$ up to height $H(x)$. Then $\Theta(\mathrm{d}T)$-a.s., the uniform measure $\bar \omega_T$ on $\partial T$ is defined as the unique probability measure on $\partial T$ satisfying that
\begin{displaymath}
\bar \omega_T(\mathcal{B}(\mathbf{v},H(x)))= \frac{e^{-H(x)}\mathcal{W}(T[x])}{\mathcal{W}(T)}, \quad\hbox{for every } x\in T \hbox{ and } \mathbf{v}\in\partial T \hbox{ such that }x\prec \mathbf{v}.
\end{displaymath}

On the other hand, for a fixed Yule-type tree $T\in\mathscr{T}$, we define the harmonic measure $\mu_T$ on $\partial T$ as the distribution of the exit ray chosen by Brownian motion on $T$ (with drift $1/2$ towards infinity).

\subsection{The invariant measure}
We write
\begin{displaymath}
\mathscr{T}^{*}\colonequals \mathscr{T}\times \{1,2\}^{\N}
\end{displaymath}
for the set of all pairs consisting of a tree $T\in \mathscr{T}$ and a distinguished geodesic ray $\mathbf{v}$. Let us define a shift transformation $S$ on $\mathscr{T}^{*}$ by shifting $(T,\mathbf{v})$ at the first branching point of $T$. More precisely, if $T$ corresponds to the collection $(z_v)_{v\in \v}$, we write $T_{(1)}$ and $T_{(2)}$ for the two subtrees of $T$ obtained at the first branching point, which means, for $i\in\{1,2\}$, $T_{(i)}$ is the tree corresponding to the collection $(z_{iv}-z_\varnothing)_{v\in \v}$. For any geodesic ray $\mathbf{v}=(v_1,v_2,\ldots)$ in the tree $T$, we set $S(T,\mathbf{v})\colonequals (T_{(v_1)}, \wt{\mathbf{v}})$, where $\wt{\mathbf{v}}=(v_2,v_3,\ldots)$.

The following proposition is the analogue of Proposition 6.1 in~\cite{LPP95} for the Yule tree.
\begin{proposition}\label{prop:unif-meas-inv}
The probability measure $\mathcal{W}(T)\Theta(\mathrm{d}T)\bar \omega_T(\mathrm{d}\mathbf{v})$ is invariant under $S$.
\end{proposition}
\begin{proof}
Under $\Theta(\mathrm{d}T)$, if $T$ corresponds to the collection $(z_v)_{v\in \v}$, then $z_{\varnothing}$ is exponentially distributed with mean 1. Conditionally on $z_{\varnothing}$, the branching property of the Yule tree states that $T_{(1)}$ and $T_{(2)}$ are i.i.d.~of the same law $\Theta$.

Let $F$ be a bounded measurable function on $\mathscr{T}^{*}$. By the definition of the shift $S$ and the preceding observation, we have
\begin{eqnarray*}
\int F\circ S (T,\mathbf{v})\mathcal{W}(T)\Theta(\mathrm{d}T) \bar\omega_T(\mathrm{d}\mathbf{v}) & =& \sum\limits_{i=1}^2 \int F(T_{(i)},\mathbf{u}) \big(e^{-z_{\varnothing}}\mathcal{W}(T_{(i)})\big) \Theta(\mathrm{d}T)\bar \omega_{T_{(i)}}(\mathrm{d}\mathbf{u})\\
&=& 2 \int_0^{\infty} e^{-2z} \,\mathrm{d}z \times \int F(T,\mathbf{u})\mathcal{W}(T) \Theta(\mathrm{d}T)\bar \omega_T(\mathrm{d}\mathbf{u}) \\
&=& \int F(T,\mathbf{u})\mathcal{W}(T) \Theta(\mathrm{d}T)\bar \omega_T(\mathrm{d}\mathbf{u}),
\end{eqnarray*}
which completes the proof.
\end{proof}

\subsection{The continuous conductances}
\label{sec:continu-conductance}

For a fixed Yule-type tree $T$, we consider the excursion measure of Brownian motion (with drift $1/2$) on $T$ away from the root, and define $\mathcal{C}(T)$ as the mass assigned by this excursion measure to the set of trajectories that never return to the root. Note that $1\leq \mathcal{C}(T)<\infty$ for any $T\in \mathscr{T}$. For more details on this probabilistic definition of the conductance $\mathcal{C}(T)$, we refer the reader to \cite[Section 2.3]{CLG13}.

To simplify notation, we introduce under the probability measure $\P$ a pair of random variables $(\mathcal{W},\mathcal{C})$ that is distributed as $(\mathcal{W}(T),\mathcal{C}(T))$ under $\Theta(\mathrm{d}T)$. In addition, we let $\widehat{\mathcal{C}}$ be a random variable defined under $\P$ that is distributed as $\mathcal{C}(T)$ under the probability measure $\mathcal{W}(T)\Theta(\mathrm{d}T)$.

Let $T$ be a Yule-type tree corresponding to the collection $(z_v)_{v\in \v}$. Recall that $T_{(1)}$ and $T_{(2)}$ stand for the two subtrees of $T$ obtained at the first branching point. From the identity $\mathcal{W}(T)= e^{-z_{\varnothing}} \big( \mathcal{W}(T_{(1)})+ \mathcal{W}(T_{(2)})\big)$ for every $T\in \mathscr{T}$, it follows that the distribution of $\mathcal{W}$ satisfies the distributional equation
$$\mathcal{W} \,\overset{(\mathrm{d})}{=\joinrel=}\, (1-U)(\mathcal{W}_{1}+\mathcal{W}_{2}),$$
in which $\mathcal{W}_1, \mathcal{W}_2$ are two independent copies of $\mathcal{W}$, and $U$ is uniformly distributed over $[0,1]$ and independent of $(\mathcal{W}_1,\mathcal{W}_2)$. Moreover, the preceding equation holds jointly with a similar distributional identity for the conductance $\mathcal{C}$ (see equation (2) in~\cite{CLG13}). To sum up, we have
\begin{equation}
\label{eq:rde-joint}
(\mathcal{W}, \mathcal{C})\,\overset{(\mathrm{d})}{=\joinrel=}\, \bigg ( (1-U)(\mathcal{W}_{1}+\mathcal{W}_{2}), \Big(U + \frac{1-U}{\mathcal{C}_1+ \mathcal{C}_2}\Big)^{-1}\bigg),
\end{equation}
where $U$ is as above, while $(\mathcal{W}_i,\mathcal{C}_i)_{i\in\{1,2\}}$ are two independent copies of $(\mathcal{W},\mathcal{C})$, and are independent of $U$.

\begin{lemma}
The random variable $\widehat{\mathcal{C}}$ satisfies the distributional identity~(\ref{eq:c*-rde}).
\end{lemma}
\begin{proof}
By definition, the law of $\widehat{\mathcal{C}}$ is determined by
\begin{equation}
\label{eq:c*-dis}
\E\big[g(\widehat{\mathcal{C}})\big]= \E\big[\mathcal{W}\,g(\mathcal{C})\big]
\end{equation}
for every nonnegative measurable function $g$. Using (\ref{eq:rde-joint}) and symmetry, we have
\begin{eqnarray*}
\E\big[g(\widehat{\mathcal{C}})\big] &=& \E\bigg[(1-U)(\mathcal{W}_{1}+\mathcal{W}_{2})\,g\Big(\Big(U + \frac{1-U}{\mathcal{C}_1+ \mathcal{C}_2}\Big)^{-1}\Big)\bigg]\\
&=& \E \bigg[2(1-U)\mathcal{W}_{1} \,g\Big(\Big(U + \frac{1-U}{\mathcal{C}_1+ \mathcal{C}_2}\Big)^{-1}\Big)\bigg].
\end{eqnarray*}
Recall that the random variable $V$ in~(\ref{eq:c*-rde}) has density function $2(1-x)$ over $[0,1]$. The statement of the lemma therefore follows by applying (\ref{eq:c*-dis}) in reverse order.
\end{proof}

The law $\gamma$ of the conductance $\mathcal{C}$ has been discussed in great detail in \cite[Proposition 6]{CLG13}. We can study the law $\widehat \gamma$ of $\widehat{\mathcal{C}}$ by similar arguments. These properties are collected in the next proposition.

For every $v \in (0,1), x\geq 0$ and $c \geq 1$, we define
\begin{displaymath}
G(v,x,c) \colonequals \left(v + \frac{1-v}{x+c}\right)^{-1}.
\end{displaymath}
Let $\mathscr{M}$ be the set of all probability measures on $[0, \infty]$ and let $\widehat \Phi \colon \mathscr{M} \to \mathscr{M}$ map a probability measure $\sigma$ to
\begin{equation}
\label{eq:iteration-map}
\widehat\Phi(\sigma) \,=\, \mathsf{Law} \big(G(V,X,\mathcal{C})\big)
\end{equation}
where $V$ and $\mathcal{C}$ are as in~(\ref{eq:c*-rde}), while $X$ is distributed according to $\sigma$, and is independent of the pair $(V,\mathcal{C})$.

\begin{proposition}
\label{prop:c*-law}
\begin{enumerate}
\item[(1)] The distributional equation~(\ref{eq:c*-rde}) characterizes the law $\widehat \gamma$ of $\widehat{\mathcal{C}}$ in the sense that, $\widehat \gamma$ is the unique fixed point of the mapping $\widehat\Phi$ on $\mathscr{M}$, and for every $\sigma \in \mathscr{M}$, the $k$-th iterate $\widehat \Phi^k(\sigma)$ converges to $\widehat \gamma$ weakly as $k \to \infty$.
\item[(2)] The law $\widehat \gamma$ has a continuous density over $[1,\infty)$, and all its moments are finite.
\item[(3)] For any monotone continuously differentiable function $g\colon [1,\infty)\to \R_{+}$, we have
\begin{equation}
\label{eq:g-c-c*}
\E\big[\widehat{\mathcal{C}}(\widehat{\mathcal{C}}-1)g'(\widehat{\mathcal{C}})\big]+2\,\E\big[g(\widehat{\mathcal{C}})\big]=2\,\E\big[g(\widehat{\mathcal{C}}+\mathcal{C})\big],
\end{equation}
where $\widehat{\mathcal{C}}$ and $\mathcal{C}$ are always assumed to be independent under the probability measure $\P$.
\item[(4)] We define, for all $\ell\geq 0$, the Laplace transforms $\varphi(\ell)=\E[\exp(-\ell\, \mathcal{C}/2)]$ and
\begin{displaymath}
\widehat \varphi(\ell)=\E[\exp(-\ell\, \widehat{\mathcal{C}}/2)]=\int_1^{\infty} e^{-\ell r/2}\,\widehat \gamma(\mathrm{d}r).
\end{displaymath}
Then $\widehat \varphi$ solves the linear differential equation
\begin{displaymath}
2\ell\, \phi''(\ell)+\ell\,\phi'(\ell)-2(1-\varphi(\ell))\phi(\ell)=0.
\end{displaymath}
\end{enumerate}
\end{proposition}

The proof is very similar to that of the analogous results in~\cite[Proposition~6]{CLG13}. We therefore skip the details.

{\bf Remark 1.} Using assertion (1) in Proposition~\ref{prop:c*-law}, one can approximate the law $\widehat \gamma$ of $\widehat{\mathcal{C}}$ by iterating the mapping $\widehat \Phi$. An application of the Monte-Carlo method gives $\E[\widehat{\mathcal{C}}]\approx 2.21$.

\smallskip
{\bf Remark 2.} Following the preceding proposition, we discuss some smoothness properties of the density of $\widehat \gamma$. For every $t\geq 1$, we set $\widehat F(t)=\P(\widehat{\mathcal{C}}\geq t)$, and we get from (\ref{eq:c*-rde}) that
\begin{equation}
\label{eq:F*}
\widehat F(t)=2\,\Big(\frac{t-1}{t}\Big)^2\int_t^{\infty} \mathrm{d}x \,\frac{x}{(x-1)^3} \,\P(\widehat{\mathcal{C}}+\mathcal{C}\geq x).
\end{equation}
Since $\P(\widehat{\mathcal{C}}+\mathcal{C}\geq t)=1$ for every $t\in [1,2]$, we obtain from the last display that
\begin{equation}
\label{eq:F*12}
\widehat F(t)= \frac{4t-2}{t^2}A_0-2A_0+1, \qquad \forall t\in[1,2],
\end{equation}
where
\begin{displaymath}
A_0 \colonequals  2-\int_2^{\infty} \mathrm{d}x \,\frac{x}{(x-1)^3} \,\P(\widehat{\mathcal{C}}+\mathcal{C}\geq x) \in \Big(\frac{1}{2},2\Big).
\end{displaymath}
Let $\widehat f=-\widehat F'$ be the density of the law $\widehat \gamma$. Then it follows from (\ref{eq:F*12}) that for all $t\in[1,2]$,
\begin{displaymath}
\widehat f(t)=4A_0 \times \frac{t-1}{t^3} \quad \mbox{ and }\quad \widehat f\,'(t)= 4A_0 \times \frac{3-2t}{t^4}.
\end{displaymath}
In particular, we have $\widehat f(1)=0$, $\widehat f(2)=A_0/2$ and $\widehat f\,'(\frac{3}{2})=0$. Numerical approximations of $\widehat f(2)$ show that $A_0\approx 0.976$.

For the density $f$ of the law of $\mathcal{C}$, it is shown in~\cite[Section 2.3]{CLG13} that there exists a constant $K_0\in (1,2)$ such that $f(t)=K_0t^{-2}$ for $t\in [1,2]$. The explicit forms of $f$ and $\widehat f$ over $[1,2]$ can be used to calculate the probability $\P(\widehat{\mathcal{C}}+\mathcal{C}\geq t)$ for $t\in [2,3]$ by convolution. The values of $\widehat F$ over $[2,3]$ are thus determined via the ordinary differential equation
\begin{equation}
\label{eq:ode-density}
t(t-1)\, \widehat F'(t)-2\,\widehat F(t)=-2\,\P(\widehat{\mathcal{C}}+\mathcal{C}\geq t),
\end{equation}
which is a direct consequence of~(\ref{eq:F*}). By solving this differential equation, we are able to get an explicit, yet complicated, expression of $\widehat F$ over $[2,3]$, in terms of the two (unknown) parameters $A_0$ and $K_0$ (numerical approximations of $f(1)$ show that $K_{0}\approx 1.477$). One can then verify that the density $\widehat f$ is continuously differentiable on $(1,3)$. Furthermore, it holds that
$$\widehat f\,''(2-)=\widehat f\,''(2+)=0\,,$$
and that $\widehat f$ is twice continuously differentiable on $(1,3)$. However, $\widehat f$ is not third-order differentiable at the point 2, as one has 
$$ \widehat f\,'''(2-)= \frac{3A_0}{4}, \quad \mbox{while } \;\;\;  \widehat f\,'''(2+)=\frac{3A_0}{4} -4 A_0 K_0.$$
This is similar to the singular behavior of the density $f$ pointed out in~\cite[Section 2.3]{CLG13}, where it is shown that $f''(2-) \neq f''(2+)$. One may conjecture that the density $\widehat f$ of $\widehat \gamma$ is twice continuously differentiable on the whole interval $(1,\infty)$, but not third-order differentiable at all integers $n\geq 2$.

We finally remark that $3/2$ is the global maximum point for the density $\widehat f$. In fact, we have seen that $\widehat f$ reaches it maximum at $3/2$ over the interval $[1,2]$. Meanwhile, it is elementary to verify, by differentiating~(\ref{eq:ode-density}), that the function $\widehat f$ is strictly decreasing over $[3/2,\infty)$.

\subsection{Proof of Theorem~\ref{thm:main-continu} and of Proposition~\ref{prop:lambda-expression}}
\label{sec:proof-continu}
We have seen in Proposition~\ref{prop:unif-meas-inv} that the probability measure $\mathcal{W}(T)\Theta(\mathrm{d}T)\bar \omega_T(\mathrm{d}\mathbf{v})$ on $\mathscr{T}^{*}$ is invariant under the shift $S$. Recall that $\mathcal{W}(T)$ follows an exponential distribution of mean~1 under $\Theta(\mathrm{d}T)$. Taking into account that $\mathcal{W}(T)>0$, $\Theta(\mathrm{d}T)$-a.s., we can then verify, in a similar way as in \cite[Proposition 2.6]{LIN}, that the shift $S$ acting on the probability space $(\mathscr{T}^{*},\mathcal{W}(T)\Theta(\mathrm{d}T)\bar \omega_T(\mathrm{d}\mathbf{v}))$ is ergodic. We shall apply Birkhoff's ergodic theorem to the three functionals defined below.

First, let $H_n(T,\mathbf{v})$ denote the height of the $n$-th branching point on the geodesic ray $\mathbf{v}$. One immediately verifies that, for every $n\geq 1$,
$$H_n= \sum_{i=0}^{n-1} H_1\circ S^i, $$
where $S^i$ stands for the $i$-th iterate of the shift $S$. It follows thus from the ergodic theorem that $\mathcal{W}(T)\Theta(\mathrm{d}T)\bar \omega_{T}(\mathrm{d}\mathbf{v})$-a.s.,
\begin{equation}
\label{ergodic11}
\frac{1}{n} H_n \build{\longrightarrow}_{n\to\infty}^{} \int H_1(T,\mathbf{v})\,\mathcal{W}(T)\,\Theta(\mathrm{d}T) \bar \omega_{T}(\mathrm{d}\mathbf{v}).
\end{equation}
Since the density $\mathcal{W}(T)$ is a.s.~strictly positive, the latter convergence also holds $\Theta(\mathrm{d}T)\bar \omega_{T}(\mathrm{d}\mathbf{v})$-a.s. By the definition of $H_1$ and then the branching property of the Yule tree,
\begin{eqnarray}
\int H_1(T,\mathbf{v})\,\mathcal{W}(T)\,\Theta(\mathrm{d}T) \bar \omega_{T}(\mathrm{d}\mathbf{v}) &= & \int z_{\varnothing} \,\mathcal{W}(T)\,\Theta(\mathrm{d}T)  \nonumber\\
&=& \int z_{\varnothing} \,e^{-z_{\varnothing}}(\mathcal{W}(T_{(1)})+\mathcal{W}(T_{(2)}))\,\Theta(\mathrm{d}T)  \nonumber  \\
&=& 2 \Big(\int_{0}^{\infty} z e^{-2z}\mathrm{d}z \Big)\times \int \mathcal{W}(T)\,\Theta(\mathrm{d}T)\, =\, \frac{1}{2}\,. \label{ergodic111}
\end{eqnarray}

Secondly, for a fixed geodesic ray $\mathbf{v}=(v_1,v_2,\ldots) \in \{1,2\}^{\N}$, we let $\mathbf{x}_{n,\mathbf{v}}$ denote the $n+$1-st branching point on the geodesic ray $\mathbf{v}$, i.e.~$\mathbf{x}_{n,\mathbf{v}}=((v_1,\ldots,v_n),H_{n+1}(T,\mathbf{v}))$. We set, for every $n\geq 1$, the functional $F_n(T,\mathbf{v}) \colonequals \log \bar \omega_T(\{\mathbf{u}\in\partial T \colon \mathbf{x}_{n,\mathbf{v}}\prec\mathbf{u}\})$. In particular,
$$F_1(T,\mathbf{v})=\log \frac{\mathcal{W}(T_{(v_1)})}{\mathcal{W}(T_{(1)})+\mathcal{W}(T_{(2)})}.$$
By the definition of $\bar \omega_T$, one can check that
$$F_n=\sum_{i=0}^{n-1} F_1\circ S^i.$$
Using the ergodic theorem again, we have $\Theta(\mathrm{d}T)\,\bar \omega_{T}(\mathrm{d}\mathbf{v})$-a.s.,
\begin{equation*}
\frac{1}{n} F_n \build{\longrightarrow}_{n\to\infty}^{} \int F_1(T,\mathbf{v})\,\mathcal{W}(T)\, \Theta(\mathrm{d}T)\bar \omega_{T}(\mathrm{d}\mathbf{v}),
\end{equation*}
in which the limit can be calculated as follows:
\begin{align}
& \;\int F_1(T,\mathbf{v})\,\mathcal{W}(T)\, \Theta(\mathrm{d}T)\bar \omega_{T}(\mathrm{d}\mathbf{v}) \nonumber\\
=\;& \;\sum_{i=1}^{2} \int e^{-z_{\varnothing}}\mathcal{W}(T_{(i)}) \log \frac{\mathcal{W}(T_{(i)})}{\mathcal{W}(T_{(1)})+\mathcal{W}(T_{(2)})}  \, \Theta(\mathrm{d}T) \nonumber\\
=\;& \;\sum_{i=1}^{2} \int e^{-z_{\varnothing}}\mathcal{W}(T_{(i)}) \log \mathcal{W}(T_{(i)}) \, \Theta(\mathrm{d}T)- \int \Big(\sum_{i=1}^{2} e^{-z_{\varnothing}}\mathcal{W}(T_{(i)}) \log\big(e^{z_{\varnothing}} \mathcal{W}(T)\big) \Big) \, \Theta(\mathrm{d}T) \nonumber\\
=\;& \;2\Big(\int_{0}^{\infty} e^{-2z}\mathrm{d}z\Big) \times \int \mathcal{W}(T) \log \mathcal{W}(T) \, \Theta(\mathrm{d}T)- \int \mathcal{W}(T) \log \big(e^{z_{\varnothing}} \mathcal{W}(T)\big)\, \Theta(\mathrm{d}T) \nonumber \\
=\;& \;-\int z_{\varnothing}\mathcal{W}(T) \, \Theta(\mathrm{d}T). \label{ergodic22}
\end{align}
Note that we used the fact that $\int \mathcal{W}(T)|\log \mathcal{W}(T)| \, \Theta(\mathrm{d}T)< \infty$ to derive the last equality. In view of (\ref{ergodic111}), we see that $\Theta(\mathrm{d}T)\,\bar \omega_{T}(\mathrm{d}\mathbf{v})$-a.s., $F_n/n$ converges to $-\frac{1}{2}$ whereas $H_n/n$ converges to~$\frac{1}{2}$. By considering the ratio $F_n/H_n$ and taking $n\to \infty$, we get that $\Theta(\mathrm{d}T)$-a.s.~$\bar\omega_T(\mathrm{d}\mathbf{v})$-a.e.,
\begin{displaymath}
\lim_{r\to \infty} \frac{1}{r}\log \bar \omega_T(\mathcal{B}(\mathbf{v},r))=-1,
\end{displaymath}
from which the convergence (\ref{eq:loc-dim-unif}) readily follows.

Thirdly, we turn to the harmonic measure $\mu_T$ and set, for every $n\geq 1$, the functional $G_n(T,\mathbf{v}) \colonequals \log \mu_T(\{\mathbf{u}\in\partial T\colon \mathbf{x}_{n,\mathbf{v}}\prec\mathbf{u}\})$. In particular,
$$G_1(T,\mathbf{v})=\log \frac{\mathcal{C}(T_{(v_1)})}{\mathcal{C}(T_{(1)})+\mathcal{C}(T_{(2)})}.$$
The flow property of the harmonic measure $\mu_T$ (see Lemma 2.3 in~\cite{LIN}) yields that
$$G_n=\sum_{i=0}^{n-1} G_1\circ S^i.$$
Similarly we have the $\Theta(\mathrm{d}T)\bar \omega_{T}(\mathrm{d}\mathbf{v})$-almost sure convergence
\begin{equation}
\label{ergodic3}
\frac{1}{n} G_n \build{\longrightarrow}_{n\to\infty}^{} \int G_1(T,\mathbf{v})\,\mathcal{W}(T)\, \Theta(\mathrm{d}T)\bar \omega_{T}(\mathrm{d}\mathbf{v}),
\end{equation}
and we calculate the limit
\begin{align*}
&\; \int G_1(T,\mathbf{v})\,\mathcal{W}(T)\, \Theta(\mathrm{d}T)\bar \omega_{T}(\mathrm{d}\mathbf{v})\\
=\;& \;\sum_{i=1}^{2}\int  e^{-z_{\varnothing}}\mathcal{W}(T_{(i)}) \log \frac{\mathcal{C}(T_{(i)})}{\mathcal{C}(T_{(1)})+\mathcal{C}(T_{(2)})}  \, \Theta(\mathrm{d}T)  \\
=\;&\;2\Big(\int_{0}^{\infty}e^{-2z}\mathrm{d}z\Big)\times \int \mathcal{W}(T_{(1)}) \log \frac{\mathcal{C}(T_{(1)})}{\mathcal{C}(T_{(1)})+\mathcal{C}(T_{(2)})} \, \Theta(\mathrm{d}T)\\
=\;&\; \int \mathcal{W}(T_{(1)}) \log \frac{\mathcal{C}(T_{(1)})}{\mathcal{C}(T_{(1)})+\mathcal{C}(T_{(2)})} \, \Theta(\mathrm{d}T).
\end{align*}
Putting the convergence (\ref{ergodic3}) together with (\ref{ergodic11}) and (\ref{ergodic111}), we see that $\Theta(\mathrm{d}T)$-a.s.~$\bar\omega_T(\mathrm{d}\mathbf{v})$-a.e.,
\begin{displaymath}
\lim_{r\to \infty} \frac{1}{r}\log  \mu_T(\mathcal{B}(\mathbf{v},r))=2\, \int \mathcal{W}(T_{(1)}) \log \frac{\mathcal{C}(T_{(1)})}{\mathcal{C}(T_{(1)})+\mathcal{C}(T_{(2)})} \, \Theta(\mathrm{d}T).
\end{displaymath}
Using the branching property of the Yule tree and recalling the notation in Section~\ref{sec:continu-conductance}, we have therefore $\P$-a.s.~$\omega(\mathrm{d}\mathbf{v})$-a.e.~that
\begin{displaymath}
\lim_{r\downarrow 0} \frac{\log \mu (\mathcal{B}_{\bd}(\mathbf{v},r))}{\log r}= 2\, \E \bigg[\log\Big(\frac{\widehat{\mathcal{C}}+\mathcal{C}}{\widehat{\mathcal{C}}}\Big)\bigg],
\end{displaymath}
where $\widehat{\mathcal{C}}$ and $\mathcal{C}$ are supposed to be independent under the probability measure $\P$. However, by taking $g(x)=\log(x)$ in (\ref{eq:g-c-c*}), we see that
\begin{equation}
\label{eq:lambda-log}
\E\big[\widehat{\mathcal{C}}\big]-1=2\,\E\bigg[\log \Big(\frac{\widehat{\mathcal{C}}+\mathcal{C}}{\widehat{\mathcal{C}}}\Big)\bigg].
\end{equation}
We define $\lambda \colonequals \E[\widehat{\mathcal{C}}]-1$. The proof of the convergence~(\ref{eq:loc-dim-harm}) is hence completed.

Finally, in view of Proposition~\ref{prop:c*-law}, it only remains to verify that $\lambda>1$. In fact, we know from the display following~(\ref{ergodic3}) that
\begin{displaymath}
\lambda =2\int e^{-z_{\varnothing}}\bigg(\sum\limits_{i=1}^2 \mathcal{W}(T_{(i)}) \log \frac{\mathcal{C}(T_{(1)})+\mathcal{C}(T_{(2)})}{\mathcal{C}(T_{(i)})}\bigg) \Theta(\mathrm{d}T).
\end{displaymath}
By concavity of the logarithm,
\begin{displaymath}
\sum\limits_{i=1}^2 \frac{\mathcal{W}(T_{(i)})}{\mathcal{W}(T_{(1)})+\mathcal{W}(T_{(2)})} \log\bigg( \frac{\mathcal{W}(T_{(1)})+\mathcal{W}(T_{(2)})}{\mathcal{W}(T_{(i)})}\cdot \frac{\mathcal{C}(T_{(i)})}{\mathcal{C}(T_{(1)})+\mathcal{C}(T_{(2)})}\bigg)\leq 0,
\end{displaymath}
which entails that
\begin{displaymath}
\sum\limits_{i=1}^2 \mathcal{W}(T_{(i)}) \log \frac{\mathcal{C}(T_{(1)})+\mathcal{C}(T_{(2)})}{\mathcal{C}(T_{(i)})} \geq \sum\limits_{i=1}^2 \mathcal{W}(T_{(i)}) \log \frac{\mathcal{W}(T_{(1)})+\mathcal{W}(T_{(2)})}{\mathcal{W}(T_{(i)})}.
\end{displaymath}
Notice that the previous inequality is strict if and only if for $i\in \{1,2\}$,
\begin{displaymath}
\frac{\mathcal{W}(T_{(i)})}{\mathcal{W}(T_{(1)})+\mathcal{W}(T_{(2)})}\neq \frac{\mathcal{C}(T_{(i)})}{\mathcal{C}(T_{(1)})+\mathcal{C}(T_{(2)})}.
\end{displaymath}
Since the latter property holds with positive probability under $\Theta(\mathrm{d}T)$, we have
\begin{align*}
\lambda \,\, >&\,\, 2\int e^{-z_{\varnothing}}\bigg(\sum\limits_{i=1}^2 \mathcal{W}(T_{(i)}) \log \frac{\mathcal{W}(T_{(1)})+\mathcal{W}(T_{(2)})}{\mathcal{W}(T_{(i)})}\bigg) \Theta(\mathrm{d}T) \\
\,=&\,\, -2 \int F_1(T,\mathbf{v})\,\mathcal{W}(T)\,\Theta(\mathrm{d}T)\,\bar \omega_{T}(\mathrm{d}\mathbf{v}).
\end{align*}
By (\ref{ergodic22}) and (\ref{ergodic111}), the right-hand side of the last display is equal to~1. We have therefore finished the proof of Theorem~\ref{thm:main-continu} and of Proposition~\ref{prop:lambda-expression}.

\subsection{The size-biased Yule tree $\widehat \Gamma$}
\label{sec:size-biased-yule}

Let $(\widehat \Gamma, \widehat{\mathbf{v}})$ denote a random element in $\mathscr{T}^{*}$ distributed according to $\mathcal{W}(T)\Theta(\mathrm{d}T)\bar \omega_T(\mathrm{d}\mathbf{v})$. We give here a direct construction of $(\widehat \Gamma, \widehat{\mathbf{v}})\in \mathscr{T}^{*}$  under the probability measure $\P$. In the following description, all the random variables involved are supposed to be defined under $\P$.

First, we introduce a sequence $(a_k)_{k\geq 1}$ of i.i.d.~random variables uniformly distributed over $\{1,2\}$, and another sequence $(J_k)_{k\geq 1}$ of i.i.d.~real random variables exponentially distributed with mean $1/2$. Let $(\Gamma^{(k)})_{k\geq 1}$ be a collection of independent Yule trees, each of which corresponding respectively to the collection $(Z^{(k)}_v)_{v\in \v}$ with the notation introduced in Section~\ref{sec:yule-tree}. We assume that $(a_k)_{k\geq 1}, (J_k)_{k\geq 1}$ and $(\Gamma^{(k)})_{k\geq 1}$ are independent.

For every integer $n\geq 1$, we set $\mathsf{v}_n=(a_1,a_2,\ldots,a_n)\in \{1,2\}^n$ and $Z_{\mathsf{v}_n}=\sum_{k=1}^n J_k$. We write $\widetilde{\mathsf{v}}_n=(a_1,a_2,\ldots,a_{n-1},3-a_n)\in \{1,2\}^n$ for the unique sibling of $\mathsf{v}_n$ in $\v$, and define the subtree $\Gamma\langle \widetilde{\mathsf{v}}_n \rangle$ grafted at $\widetilde{\mathsf{v}}_n$ as
$$\Gamma\langle \widetilde{\mathsf{v}}_n \rangle  \colonequals \Big(\{\widetilde{\mathsf{v}}_n \}\times (Z_{\mathsf{v}_n},Z_{\mathsf{v}_n}+Z^{(n)}_\varnothing]\Big) \cup  \Bigg(\bigcup_{v\in\mathcal{V}\backslash\{\varnothing\}} \{\widetilde{\mathsf{v}}_n v\} \times (Z_{\mathsf{v}_n}+Z^{(n)}_{\bar v}, Z_{\mathsf{v}_n}+ Z^{(n)}_v]\Bigg).$$
Finally, let $\widehat \Gamma$ be the following Yule-type tree
$$\widehat \Gamma \colonequals (\{\varnothing\}\times [0,Z_{\mathsf{v}_1}]) \cup  \bigg(\bigcup_{n\geq 1} \{\mathsf{v}_n\} \times (Z_{\mathsf{v}_n}, Z_{\mathsf{v}_{n+1}}]\bigg) \cup \bigg(\bigcup_{n\geq 1} \Gamma\langle \widetilde{\mathsf{v}}_n \rangle \bigg).$$
We will call $\widehat \Gamma$ the size-biased Yule tree. See Figure~\ref{figure:yulebias} for an illustration.

\begin{figure}[!h]
\begin{center}
\includegraphics[width=7.5cm]{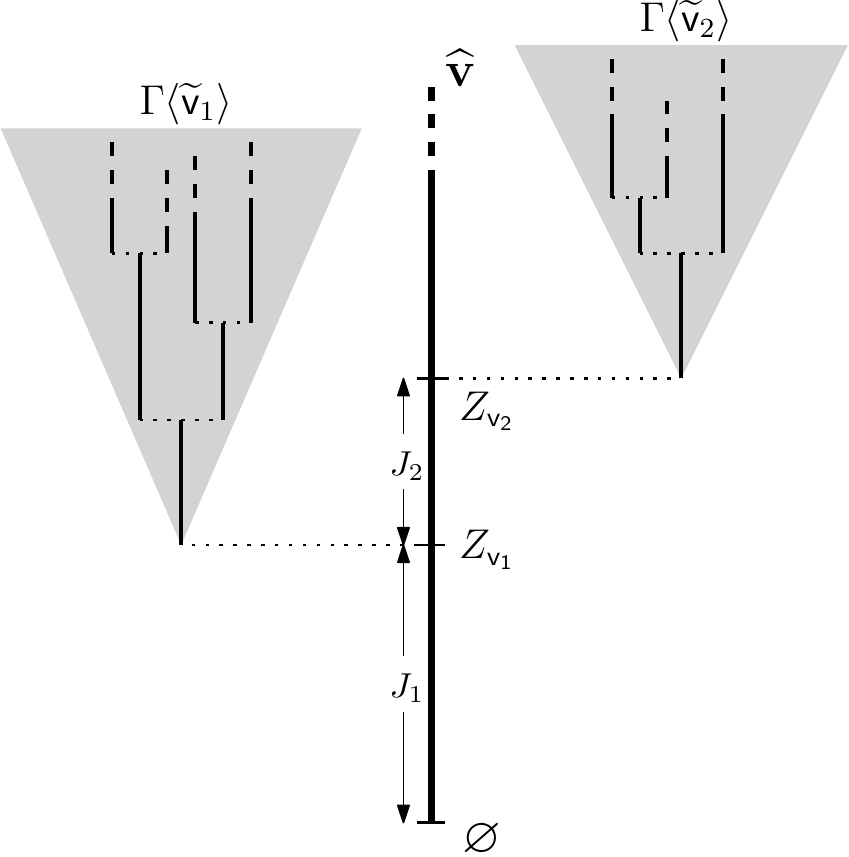}
\caption{Schematic representation of the size-biased Yule tree $\widehat \Gamma$ \label{figure:yulebias}}
\end{center}
\end{figure}

\begin{lemma}
\label{lem:yulebias}
The pair $\big(\widehat \Gamma, \widehat{\mathbf{v}}=(a_1,a_2,\ldots)\big)\in \mathscr{T}\times \{1,2\}^{\N} $ constructed above follows the required distribution $\mathcal{W}(T)\Theta(\mathrm{d}T)\bar \omega_T(\mathrm{d}\mathbf{v})$.
\end{lemma}

The proof of this lemma is based on similar calculations carried out in the previous section. We leave the details to the reader. Note that the analog for discrete-time Galton--Watson trees can be found in Exercise 16.9 in \cite{LP10}.

\subsection{The size-biased reduced tree $\widehat \Delta$}
\label{sec:size-biased-reduced}
Recall the bijection $\Psi \colon (v,r)\in \Delta_0 \mapsto (v,-\log (1-r))\in \Gamma$ introduced in Section~\ref{sec:yule-tree}. We now apply the inverse mapping $\Psi^{-1}(v,s)=(v,1-e^{-s})$ to the size-biased Yule tree $\widehat \Gamma$.

We keep the notation of the preceding section. For every integer $n\geq 1$, we set $V_n =1-\exp(-J_n)$, and then by induction,
\begin{displaymath}
\widehat Y_{\mathsf{v}_n}= \widehat Y_{\mathsf{v}_{n-1}}+(1-\widehat Y_{\mathsf{v}_{n-1}})V_n \,.
\end{displaymath}
Notice that $(V_n)_{n\geq 1}$ are i.i.d.~real random variables with density function $2(1-x)\mathbf{1}_{x\in[0,1]}$, and that for every $n\geq 1$, $\widehat Y_{\mathsf{v}_n}=1-\exp(-Z_{\mathsf{v}_n})$. Thus,
$$\Psi^{-1}(\widehat \Gamma)=(\{\varnothing\}\times [0,\widehat Y_{\mathsf{v}_1}]) \cup  \bigg(\bigcup_{n\geq 1} \{\mathsf{v}_n\} \times (\widehat Y_{\mathsf{v}_n}, \widehat Y_{\mathsf{v}_{n+1}}]\bigg) \cup \bigg(\bigcup_{n\geq 1} \Psi^{-1}\big(\Gamma\langle \widetilde{\mathsf{v}}_n \rangle \big)\bigg).$$
We point out that, independently for every $n\geq 1$, $\Psi^{-1}\big( \Gamma \langle \widetilde{\mathsf{v}}_n \rangle \big)$ is a rescaled copy of the precompact reduced tree $\Delta_0$ with the scaling factor $(1-\widehat Y_{\mathsf{v}_n})$. From now on we will denote $\Psi^{-1}(\widehat \Gamma)$ by $\widehat \Delta_0$. See Figure~\ref{figure:reducedbias} for an illustration of $\widehat \Delta_0$.

\begin{figure}[!h]
\begin{center}
\includegraphics[width=11cm]{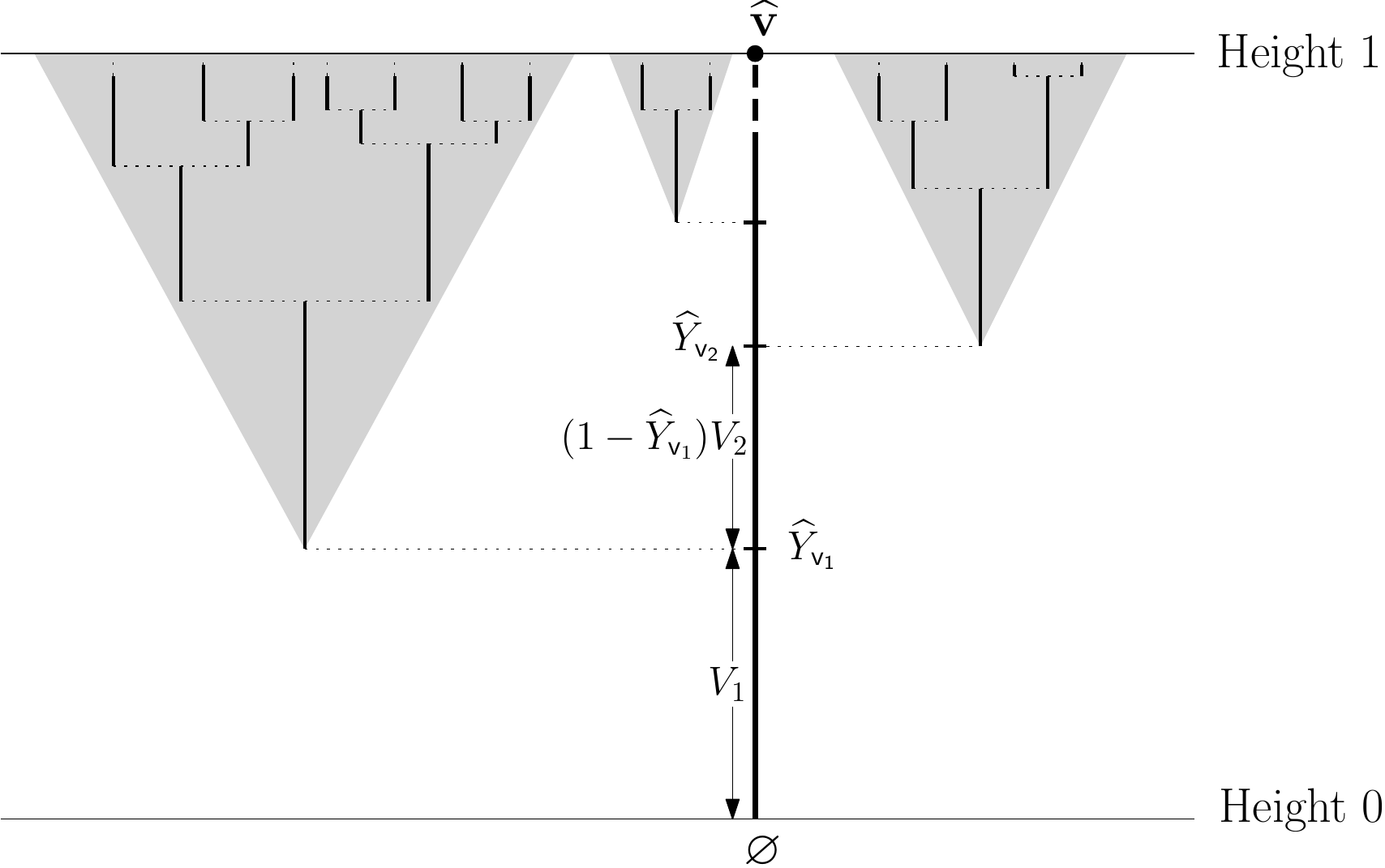}
\caption{Schematic representation of the random tree $\widehat \Delta_0$ \label{figure:reducedbias}}
\end{center}
\end{figure}

One can define, as for $\Delta_0$, the intrinsic metric $\mathbf{d}$ on $\widehat \Delta_0$ such that $(\widehat \Delta_0, \mathbf{d})$ is a noncompact $\mathbb{R}$-tree, and for every $x=(v,r)\in \widehat \Delta_0$, we have $\mathbf{d}((\varnothing,0),x)=r$. Then we let $\widehat \Delta$ be the completion of $\widehat \Delta_0$ with respect to $\mathbf{d}$, so that $(\widehat \Delta, \mathbf{d})$ is a compact $\R$-tree. In fact, $\widehat \Delta=\widehat \Delta_0\cup \partial \widehat \Delta$, and the boundary $\partial \widehat \Delta \colonequals \{x\in \widehat \Delta\colon \mathbf{d}((\varnothing,0),x)=1\}$ is canonically identified with $\{1,2\}^{\N}$. We will call $\widehat \Delta$ the size-biased reduced tree. We keep the same notation $H(x)=\mathbf{d}((\varnothing,0),x)$ for the height of $x\in\widehat \Delta$. For every $\varepsilon\in (0,1)$, we set the truncation of $\widehat\Delta$ at height $1-\ve$
$$\widehat\Delta_\ve \colonequals \{x\in\widehat \Delta\colon H(x)\leq 1-\ve\}.$$

One can think of both $\Delta$ and $\widehat \Delta$ as electric networks of ideal resistors with unit resistance per unit length, and define $\mathcal{C}(\Delta)$ (resp.~$\mathcal{C}(\widehat \Delta)$) be the effective conductance between the root and the set $\partial \Delta$ (resp.~$\partial \widehat \Delta$) in the corresponding network. As explained in~\cite[Section 2.3]{CLG13}, $\mathcal{C}(\Delta)$ is identically distributed as the random variable $\mathcal{C}$ defined in Section~\ref{sec:continu-conductance}. Analogously, $\mathcal{C}(\widehat \Delta)$ has the same distribution as $\widehat{\mathcal{C}}$ according to Lemma~\ref{lem:yulebias}. We thus call $\widehat{\mathcal{C}}$ the continuous conductance of the size-biased reduced tree $\widehat \Delta$.

\section{The discrete setting}
\label{sec:discrete-setting-typical}

\subsection{Notation for discrete trees}
\label{sec:notation-discrete-tree}
We set
$$\mathcal{U} \colonequals\bigcup_{n=0}^\infty \N^n,$$
where $\N=\{1,2,\ldots\}$ and $\N^0=\{\varnothing\}$ by convention. If $u=(u_1,\ldots,u_n)\in\mathcal{U}$, $|u|=n$ is called the generation (or height) of $u$. In particular, $|\varnothing|=0$. A (rooted ordered) tree $\mathcal{T}$ is a subset of $\mathcal{U}$ such that the following holds:
\begin{enumerate}
\item[(i)] $\varnothing\in \mathcal{T}$;

\item[(ii)] If $u=(u_1,\ldots,u_n)\in \mathcal{T} \backslash\{\varnothing\}$, then $\bar u\colonequals (u_1,\ldots,u_{n-1})\in \mathcal{T}$;

\item[(iii)] For every $u=(u_1,\ldots,u_n)\in \mathcal{T}$, there exists an integer $k_u(\mathcal{T})\geq 0$ such that, for every $j\in\N$, $(u_1,\ldots,u_n,j)\in \mathcal{T}$ if and only if $1\leq j\leq k_u(\mathcal{T})$.
\end{enumerate}
The notions of a child and a parent of a vertex of $\mathcal{T}$ are defined in an obvious way. We write~$\prec$ for the genealogical order on $\mathcal{T}$. The quantity $k_u(\mathcal{T})$ in (iii) is called the number of children of~$u$ in $\mathcal{T}$. We always view a tree $\mathcal{T}$ as a graph whose vertices are the elements of $\mathcal{T}$ and whose edges are the pairs $\{\bar u, u\}$ for all $u\in \mathcal{T}\backslash \{\varnothing\}$.

If $\mathcal{T}$ is finite, we call it a plane tree. The set of all plane trees is denoted by $\mathbb{T}_f$. For an infinite tree $\mathcal{T}$, we say it has a single infinite line of descent if there exists a unique sequence of positive integers $(u_n)_{n\geq 1}$ such that $(u_1,u_2,\ldots,u_n)\in \mathcal{T}$ for all $n\geq 1$. We denote by $\mathbb{T}_{\infty}$ the set of all infinite trees that have a single infinite line of descent.

The height of a tree $\mathcal{T}$ is written as $h(\mathcal{T})\colonequals \sup\{|u|\colon u\in \mathcal{T}\}$. The set of all vertices of $\mathcal{T}$ at generation~$n$ is denoted by $\mathcal{T}_n\colonequals \{u \in \mathcal{T} \colon |u|=n\}$. If $u\in \mathcal{T}$, the subtree of descendants of $u$ is $\widetilde{\mathcal{T}}[u]\colonequals \{u'\in \mathcal{T} \colon u\prec u'\}$. Note that $\widetilde{\mathcal{T}}[u]$ is not a tree under our definition, but we can relabel its vertices to turn it into a tree, by setting $\mathcal{T}[u]\colonequals \{w\in \mathcal{U}\colon uw\in \mathcal{T}\}$.

Let $\mathcal{T}$ be a tree of height larger than $n$, and consider a simple random walk $X=(X_k)_{k\geq 0}$ on $\mathcal{T}$ starting from the root $\varnothing$, which is defined under the probability measure~$P^{\mathcal{T}}$. We write $\tau_n\colonequals \inf\{k\geq 0 \colon |X_k|=n\}$ for the first hitting time of generation $n$ by $X$, and we define the discrete harmonic measure $\mu_n^{\mathcal{T}}$ supported on $\mathcal{T}_n$ as the law of $X_{\tau_n}$ under $P^{\mathcal{T}}$.

\smallskip
{\bf Critical Galton--Watson trees.} Let $\theta$ be a non-degenerate probability measure on $\mathbb{Z}_{+}$, and assume that $\theta$ has mean one and finite variance $\sigma^2>0$. For every integer $n\geq 0$, we let $\mathsf{T}^{(n)}$ be a Galton--Watson tree with offspring distribution $\theta$, conditioned on non-extinction at generation $n$, viewed as a random element in $\mathbb{T}_{f}$. In particular, $\mathsf{T}^{(0)}$ is just a Galton--Watson tree with offspring distribution $\theta$. We suppose that the random trees $\mathsf{T}^{(n)}$ are defined under the probability measure $\P$.

Let $\mathsf{T}^{*n}$ be the reduced tree associated with $\mathsf{T}^{(n)}$, which is the random tree composed of all vertices of $\mathsf{T}^{(n)}$ that have descendants at generation $n$. It is always implicitly assumed that we have relabeled the vertices of $\mathsf{T}^{*n}$, preserving both the lexicographical order and the genealogical order, so that $\mathsf{T}^{*n}$ becomes a plane tree in the sense of our preceding definition.

For every $n\geq 1$ we set $q_n\colonequals \P(\#\mathsf{T}_n^{(0)}>0)$. By a standard result (see e.g.~Theorem 9.1~of~\cite[Chapter 1]{AN}) on the non-extinction probability up to generation $n$, we have
\begin{equation}
\label{eq:kolmogorov}
q_n \sim \frac{2}{n \sigma^2}\,, \qquad \mbox{ as } n \to \infty.
\end{equation}

{\bf Size-biased Galton--Watson tree.} Let $\widehat N$ be a random variable distributed according to the size-biased distribution of $\theta$, that is, for every $k\geq 0$, $\P(\widehat N=k)=k\,\theta(k)$. Take a sequence $(\widehat N_k)_{k\geq 1}$ of independent copies of $\widehat N$ defined under~$\P$. Now we follow Kesten~\cite{Kes86} and Lyons, Pementle and Peres~\cite{LPP95b} to construct a size-biased Galton--Watson tree $\widehat{\mathsf{T}}$ defined under~$\P$. First, the root $\varnothing$ of $\widehat{\mathsf{T}}$ is given a number $\widehat N_1$ of children. Choose one of these children uniformly at random, say $\mathbf{v}_1$. It has a number $\widehat N_2$ of children, whereas the other children of the root have independently ordinary $\theta$-Galton--Watson descendant trees. Again, among the children of $\mathbf{v}_1$ we choose one uniformly at random, call it $\mathbf{v}_2$, and give the others independent $\theta$-Galton--Watson descendant trees. Meanwhile the vertex $\mathbf{v}_2$ has a number $\widehat N_3$ of children. Since a.s.~$\widehat N\geq 1$, we can repeat this procedure infinitely many times. The resulting random infinite tree $\widehat{\mathsf{T}}$ is called a size-biased Galton--Watson tree  (see Figure~\ref{figure:gwbias}). It is clear by the construction that $\widehat{\mathsf{T}}$ is a random element in $\mathbb{T}_{\infty}$ and that its unique infinite line of descent is $(\mathbf{v}_1,\mathbf{v}_2,\ldots)$, which we will call the spine of $\widehat{\mathsf{T}}$.

\begin{figure}[!h]
\begin{center}
\includegraphics[width=9cm]{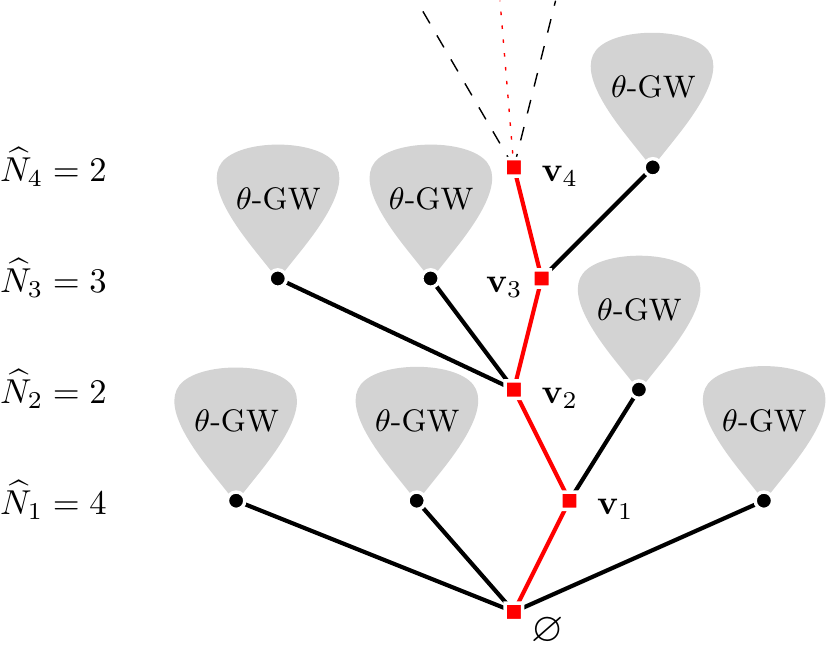}
\caption{Schematic representation of a size-biased Galton--Watson tree $\widehat{\mathsf{T}}$ \label{figure:gwbias}}
\end{center}
\end{figure}

Let $[\widehat{\mathsf{T}}]^{(n)}$ be the plane tree obtained from $\widehat{\mathsf{T}}$ by keeping only its first $n$ generations, i.e.,
$$[\widehat{\mathsf{T}}]^{(n)} \colonequals \{u\in \widehat{\mathsf{T}}\colon |u|\leq n\}.$$
It is shown in~\cite{Kes86} and~\cite{LPP95b} that $[\widehat{\mathsf{T}}]^{(n)}$ is distributed in $\mathbb{T}_f$ according to the law of $\mathsf{T}^{(n)}$ biased by $\#\mathsf{T}^{(n)}_n$. Moreover, conditionally given the first $n$ levels of $\widehat{\mathsf{T}}$, the vertex $\mathbf{v}_n$ on the spine is uniformly distributed on the $n$-th level of $\widehat{\mathsf{T}}$. Besides, notice that
\begin{displaymath}
\E\big[\# \mathsf{T}_n^{(n)}\big]=\frac{\E[\# \mathsf{T}_n^{(0)}]}{q_n}=\frac{1}{q_n}.
\end{displaymath}
All these observations are summarized in the following proposition.

\begin{proposition}\label{prop:size-bias-palm}
Let $F(\mathcal{T},v)$ be a nonnegative measurable function defined on $\mathbb{T}_f\times \mathcal{U}$. Then for every integer $n\geq 1$,
\begin{displaymath}
\E \Bigg[\sum\limits_{v\in \mathsf{T}_n^{(n)}}^{} F(\mathsf{T}^{(n)},v) \Bigg]= \frac{1}{q_n} \E\Big[F([\widehat{\mathsf{T}}]^{(n)}, \mathbf{v}_n)\Big].
\end{displaymath}
\end{proposition}

For every integer $n\geq 1$, let $[\widehat{\mathsf{T}}]^n$ be the plane tree obtained from $\widehat{\mathsf{T}}$ by erasing the (infinite) tree of descendants of the vertex $\mathbf{v}_n$. By convention, the vertex $\mathbf{v}_n$ is kept in $[\widehat{\mathsf{T}}]^n$. Notice that in general $[\widehat{\mathsf{T}}]^n \neq [\widehat{\mathsf{T}}]^{(n)}$, since the height of $[\widehat{\mathsf{T}}]^n$ can be strictly greater than $n$.

We let $[\widehat{\mathsf{T}}]^{*n}$ denote the reduced tree associated with the plane tree $[\widehat{\mathsf{T}}]^n$ up to generation~$n$, which consists of all vertices of $[\widehat{\mathsf{T}}]^n$ that have (at least) one descendant at generation $n$. We implicitly assume that the relabelling has been done to turn $[\widehat{\mathsf{T}}]^{*n}$ into a tree. It is elementary to check that $[\widehat{\mathsf{T}}]^{*n}$ is also the reduced tree associated with $[\widehat{\mathsf{T}}]^{(n)}$ up to generation $n$.

\subsection{Convergence of discrete reduced trees}
\label{sec:conv-discr-reduc}

We briefly recall the result in \cite{CLG13} on the convergence of the discrete reduced trees $\mathsf{T}^{*n}$. For every real number $s\in[0,n]$, we write the truncation of the tree $\mathsf{T}^{*n}$ at level $n-\lfloor s\rfloor$ as
$$R_s(\mathsf{T}^{*n})\colonequals \big\{u\in \mathsf{T}^{*n}\colon |u|\leq n-\lfloor s\rfloor\big\}.$$

For every $\ve\in(0,1)$, we have set $\Delta_\ve=\{x\in\Delta\colon H(x)\leq 1-\ve\}$. We know that, for every fixed $\ve$, there is a.s.~no branching point of $\Delta$ at height $1-\ve$. The skeleton of $\Delta_\ve$ is defined as the following plane tree
\begin{displaymath}
\hbox{Sk}(\Delta_\ve)\colonequals \{\varnothing\}\cup\big\{v\in \mathcal{V}\backslash\{\varnothing\}\colon Y_{\bar v}\leq 1-\ve\big\}= \{\varnothing\}\cup \big\{v\in \mathcal{V}\backslash\{\varnothing\}\colon (\bar v,Y_{\bar v})\in \Delta_\ve\big\}.
\end{displaymath}

Consider then the set $\mathbb{T}_{f,bin}$ of all plane trees in which every vertex has either 0, 1 or 2 children. For $\mathcal{T}\in \mathbb{T}_{f,bin}$ we write $\mathcal{S}(\mathcal{T})$ for the set of all vertices of $\mathcal{T}$ having 0 or 2 children. Then there is a unique plane tree $\langle\mathcal{T}\rangle$ such that one can find a canonical bijection $u\mapsto w_u$ from $\langle\mathcal{T}\rangle$ onto $\mathcal{S}(\mathcal{T})$ that preserves the genealogical order and the lexicographical order of vertices.

The following result is Proposition 16 in \cite{CLG13}.
\begin{proposition}
\label{prop:conv-reduced-tree}
We can construct the reduced trees $\mathsf{T}^{*n}$ and the (continuous) tree $\Delta$ on the same probability space $(\Omega,\mathcal{F},\P)$ so that the following properties hold for every fixed $\ve\in(0,1)$ with $\P$-probability one.
\begin{enumerate}
\item[\rm(i)] For every
sufficiently large integer $n$, we have $R_{\ve n}(\mathsf{T}^{*n})\in \mathbb{T}_{f,bin}$ and $\langle R_{\ve n}(\mathsf{T}^{*n})\rangle  =\hbox{\rm Sk}(\Delta_\ve)$.
\item[\rm(ii)] For every sufficiently large $n$, such that the properties stated in {\rm (i)} hold, and for every $u\in \hbox{\rm Sk}(\Delta_\ve)$, let $w^{n,\ve}_u$ denote the  vertex of $\mathcal{S}(R_{\ve n}(\mathsf{T}^{*n}))$ corresponding to $u$ via the canonical bijection from $\langle R_{\ve n}(\mathsf{T}^{*n})\rangle $ onto $\mathcal{S}(R_{\ve n}(\mathsf{T}^{*n}))$. Then we have
$$\lim_{n\to\infty} \frac{1}{n} |w^{n,\ve}_u| = Y_{u} \wedge (1-\ve).$$
\end{enumerate}
\end{proposition}

\subsection{Convergence of discrete conductances}
\label{sec:cv-dis-cond}
Let $\mathcal{T}\in \mathbb{T}_f$ be a tree of height larger than $n$, and consider the new tree $\mathcal{T}'$ obtained by adding to the graph $\mathcal{T}$ an edge between the root $\varnothing$ and an extra vertex $\partial$. We define under the probability measure $P^{\mathcal{T}'}$ a simple random walk $X$ on $\mathcal{T}'$ starting from the root $\varnothing$. Let $\tau_{\partial}$ be the first hitting time of $\partial$ by $X$, and for every integer $1\leq i \leq n$, let $\tau_i$ be the first hitting time of generation $i$ (of the tree $\mathcal{T}$) by $X$. We write
$$\mathcal{C}_i(\mathcal{T})\colonequals P^{\mathcal{T}'}(\tau_i<\tau_{\partial}).$$
This notation is justified by the fact that $\mathcal{C}_i(\mathcal{T})$ can be interpreted as the effective conductance between $\partial$ and generation $i$ of $\mathcal{T}$ in the graph $\mathcal{T}'$, see e.g.~\cite[Chapter 2]{LP10}.

Recall the notation that $\mathcal{C}(\Delta)$ stands for the conductance between the root and the set $\partial \Delta$ in the reduced tree $\Delta$. Analogously, for every $\varepsilon \in(0,1)$, $\mathcal{C}(\Delta_{\varepsilon})$ denotes the conductance between the root and the set $\{x\in \Delta\colon H(x)=1-\varepsilon\}$ in $\Delta$. The following result is stated in \cite{CLG13} without a proof. We provide here the details.

\begin{proposition}\label{prop:cv-conductance}
Suppose that the reduced trees $\mathsf{T}^{*n}$ and the tree $\Delta$ are constructed so that the properties stated in Proposition~\ref{prop:conv-reduced-tree} hold. Then
\begin{displaymath}
n\,\mathcal{C}_n(\mathsf{T}^{*n}) \xrightarrow[n\to\infty]{\mathrm{a.s.}}  \mathcal{C}(\Delta).
\end{displaymath}
\end{proposition}

\begin{proof}
By definition, for every $\varepsilon \in (0,1)$,
\begin{displaymath}
\mathcal{C}_{n-\lfloor \varepsilon n\rfloor}(\mathsf{T}^{*n})=P^{(\mathsf{T}^{*n})'}(\tau_{n-\lfloor \varepsilon n\rfloor}<\tau_{\partial}) \geq P^{(\mathsf{T}^{*n})'}(\tau_n<\tau_{\partial})=\mathcal{C}_n(\mathsf{T}^{*n}).
\end{displaymath}
Note that there is probability at least $1-\frac{\lfloor \varepsilon n\rfloor}{n+1}$ that, after hitting the generation $n-\lfloor \varepsilon n\rfloor$, the simple random walk on $(\mathsf{T}^{*n})'$ will hit the generation $n$ before moving down to the extra vertex~$\partial$. Hence it follows from the strong Markov property of simple random walk that
\begin{displaymath}
0\leq  \mathcal{C}_{n-\lfloor \varepsilon n\rfloor}(\mathsf{T}^{*n}) - \mathcal{C}_n(\mathsf{T}^{*n}) \leq  \frac{\lfloor \varepsilon n\rfloor}{n+1} \mathcal{C}_{n-\lfloor \varepsilon n\rfloor}(\mathsf{T}^{*n}).
\end{displaymath}
By similar probabilistic arguments, we also have
\begin{displaymath}
0\leq \mathcal{C}(\Delta_{\varepsilon})- \mathcal{C}(\Delta)\leq \varepsilon\, \mathcal{C}(\Delta_{\varepsilon}),
\end{displaymath}
which entails particularly that $\mathcal{C}(\Delta_{\varepsilon})\leq 2\, \mathcal{C}(\Delta)$ if $\varepsilon< 1/2$.

Let $n$ be sufficiently large so that assertions (i) and (ii) of Proposition~\ref{prop:conv-reduced-tree} hold with $\varepsilon \in (0,\frac{1}{2})$. By calculating the conductances using the series law and parallel law, we see that a.s.
\begin{displaymath}
\lim_{n\to \infty} |n\,\mathcal{C}_{n-\lfloor \varepsilon n\rfloor}(\mathsf{T}^{*n})-\mathcal{C}(\Delta_{\varepsilon})|=0.
\end{displaymath}
Then, it follows from
\begin{align*}
|n\,\mathcal{C}_n(\mathsf{T}^{*n})-\mathcal{C}(\Delta)| & \leq  |n\,\mathcal{C}_n(\mathsf{T}^{*n})-n\,\mathcal{C}_{n-\lfloor \varepsilon n\rfloor}(\mathsf{T}^{*n})|+|n\,\mathcal{C}_{n-\lfloor \varepsilon n\rfloor}(\mathsf{T}^{*n})-\mathcal{C}(\Delta_{\varepsilon})|+|\mathcal{C}(\Delta_{\varepsilon})-\mathcal{C}(\Delta)|\\
& \leq  \lfloor \varepsilon n\rfloor \mathcal{C}_{n-\lfloor \varepsilon n\rfloor}(\mathsf{T}^{*n})+ |n\,\mathcal{C}_{n-\lfloor \varepsilon n\rfloor}(\mathsf{T}^{*n})-\mathcal{C}(\Delta_{\varepsilon})| +\varepsilon\, \mathcal{C}(\Delta_{\varepsilon})
\end{align*}
that
\begin{displaymath}
\limsup _{n \to \infty} |n\,\mathcal{C}_n(\mathsf{T}^{*n})-\mathcal{C}(\Delta)| \leq 2\varepsilon\,\mathcal{C}(\Delta_{\varepsilon}) \leq 4\varepsilon\, \mathcal{C}(\Delta).
\end{displaymath}
Letting $\varepsilon\to 0$, we conclude that $|n\,\mathcal{C}_n(\mathsf{T}^{*n})-\mathcal{C}(\Delta)|\to 0$ as $n\to \infty$.
\end{proof}

Let us write $\mathcal{C}(\widehat \Delta_{\varepsilon})$ for the conductance between the root and the set $\{x\in \widehat \Delta \colon H(x)=1-\varepsilon\}$
in $\widehat \Delta$. By the same reasoning as in the previous proof, for every $\varepsilon\in(0,1/2)$, we have
\begin{equation}
\label{eq:cond-continu-epsilon}
0\leq  \mathcal{C}(\widehat \Delta_{\varepsilon})- \mathcal{C}(\widehat \Delta)\leq 2\varepsilon \, \mathcal{C}(\widehat \Delta).
\end{equation}

\remark We can also show the convergence of the reduced trees $[\widehat{\mathsf{T}}]^{*n}$ to the (continuous) size-biased reduced tree $\widehat \Delta$, in a sense similar to Proposition \ref{prop:conv-reduced-tree}, which implies a size-biased analog of Proposition~\ref{prop:cv-conductance}. In particular, it holds the convergence in law
\begin{displaymath}
n\,\mathcal{C}_n([\widehat{\mathsf{T}}]^n) \xrightarrow[n\to\infty]{\mathrm{(d)}}  \mathcal{C}(\widehat \Delta).
\end{displaymath}
Since these results are not needed in Section \ref{sec:proof-thm1} for proving Theorem~\ref{thm:main-dis}, we omit the proofs. 

\smallskip
For future reference, we state the following result, which is Lemma~22 in \cite{CLG13}.

\begin{lemma}
\label{lem:L2-gw-condct}
There exists a constant $K\geq 1$ such that, for every integer $n\geq 1$,
\begin{displaymath}
\E\Big[\big(n\,\mathcal{C}_n(\mathsf{T}^{*n})\big)^2\Big]\leq K.
\end{displaymath}
\end{lemma}

\subsection{Backward size-biased Galton--Watson tree}
\label{sec:backward-size-biased}

We introduce in this section a new infinite random tree $\widecheck{\mathsf{T}}$, which originates from the inflated Galton--Watson tree constructed by Peres and Zeitouni in~\cite{PZ08}. It will be clear from the following descriptions that $\widecheck{\mathsf{T}}$ is a rear-view variant of the size-biased Galton--Watson tree $\widehat{\mathsf{T}}$.

First, the random tree $\widecheck{\mathsf{T}}$ has a unique infinite ray of vertices $(\mathbf{u}_0,\mathbf{u}_1,\mathbf{u}_2,\ldots)$, which will be referred to as its spine. We declare that, for every $n\geq 0$, the vertex $\mathbf{u}_n$ is at generation $-n$. This gives a genealogical order on the spine of $\widecheck{\mathsf{T}}$: $\mathbf{u}_{1}$ is viewed as the parent of $\mathbf{u}_0$, $\mathbf{u}_2$ is viewed as the parent of $\mathbf{u}_1$, and so on.

Next, let us describe the finite subtrees in $\widecheck{\mathsf{T}}$ branching off every node of the spine. To this end, we recall that $\widehat N$ follows the size-biased distribution of $\theta$, and we denote by $\mathsf{L}$ a random variable which, conditionally on $\widehat N$, is uniformly distributed on the set $\{0,1,\ldots,\widehat N-1\}$. Let $(\mathsf{L}_n,\widehat N_n)_{n\geq 1}$ be a sequence of i.i.d.~copies of $(\mathsf{L},\widehat N)$, and set $\mathsf{R}_n=\widehat N_n-\mathsf{L}_n-1$ for every $n\geq 1$. To every vertex $\mathbf{u}_n$ we give a number $\mathsf{L}_n$ of children to the left of the spine and a number $\mathsf{R}_n$ of children to the right of the spine. Each of these children (there are $\widehat N_n-1$ in total) will independently have an ordinary $\theta$-Galton--Watson descendant tree (later we will say that these Galton--Watson trees are grafted at $\mathbf{u}_n$), and we also assume the independence of these Galton--Watson trees among all $n\geq 1$. This finishes the construction of $\widecheck{\mathsf{T}}$. See Figure~\ref{figure:backgwbias} for an illustration. We remark that $\widecheck{\mathsf{T}}$ is not a tree in the sense of Section~\ref{sec:notation-discrete-tree}. However, due to its obvious tree structure, we will call $\widecheck{\mathsf{T}}$ the backward size-biased Galton--Watson tree.

\begin{figure}[!h]
\begin{center}
\includegraphics[width=12cm]{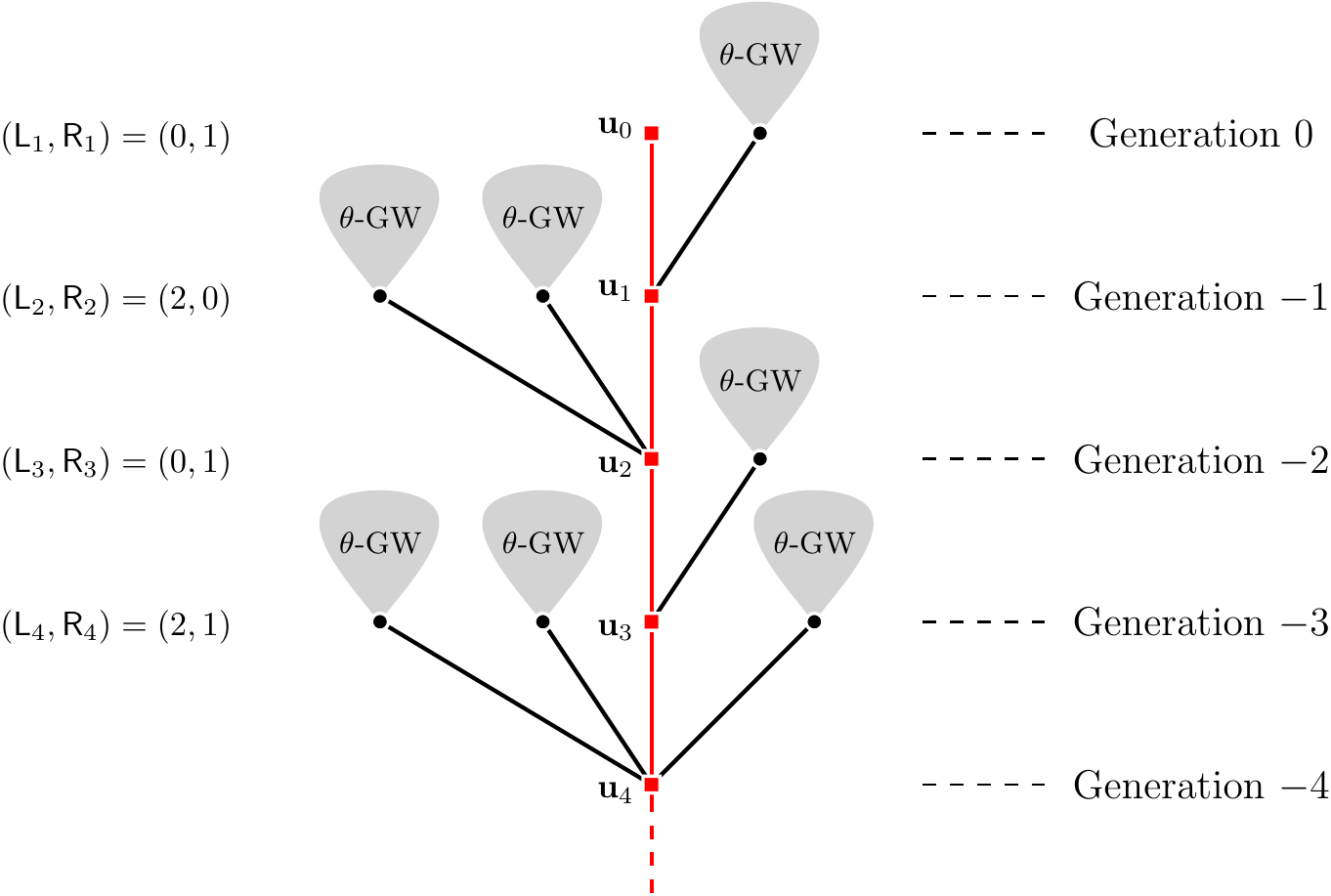}
\caption{Schematic representation of the backward size-biased Galton--Watson tree~$\check{\mathsf{T}}$ \label{figure:backgwbias}}
\end{center}
\end{figure}

The genealogical (partial) order on $\widecheck{\mathsf{T}}$ is defined in the following way. We simply keep the genealogical orders inherited from the grafted Galton--Watson trees and combine them with the genealogical order on the spine. For instance, $\mathbf{u}_2$ is an ancestor of any vertex in the subtrees grafted at $\mathbf{u}_1$. We can also define in a consistent manner the notion of generation for every vertex in $\widecheck{\mathsf{T}}$. In fact, for any vertex $v$ not on the spine, there is a unique vertex $\mathbf{u}_{m}$ on the spine such that $v$ belongs to a finite subtree grafted at $\mathbf{u}_{m}$, then we say that the generation of $v$ in $\widecheck{\mathsf{T}}$ is equal to $-m+1$ plus the initial generation of $v$ inside the corresponding grafted plane tree.

The vertex $\mathbf{u}_0$ in $\widecheck{\mathsf{T}}$ corresponds to the root in the inflated Galton--Watson tree (abbreviated as \emph{IGW}) described in Peres and Zeitouni~\cite{PZ08}. It is worth pointing out that, while an independent number of offspring of distribution $\theta$ is assigned the root of \emph{IGW}, the vertex $\mathbf{u}_0$ in $\widecheck{\mathsf{T}}$ has no descendants. Besides, our notion of generation on $\widecheck{\mathsf{T}}$ can also be called as the horocycle distance from $\mathbf{u}_0$, if we follow the terminology in~\cite{PZ08}.

For every $n\geq 1$, let $[\widecheck{\mathsf{T}}]^{n}$ be the plane tree obtained from $\widecheck{\mathsf{T}}$ by only keeping the finite tree above the vertex $\mathbf{u}_n$. We take $\mathbf{u}_n$ as the root of $[\widecheck{\mathsf{T}}]^{n}$, and the lexicographical order on the set of vertices of $[\widecheck{\mathsf{T}}]^{n}$ corresponds to the order of visit when one ``moves around'' the finite tree $[\widecheck{\mathsf{T}}]^{n}$ in clockwise order, starting from the root $\mathbf{u}_n$. A key observation is that, viewed as a random plane tree, $[\widecheck{\mathsf{T}}]^{n}$ has the same distribution as the random tree $[\widehat{\mathsf{T}}]^{n}$ defined in Section~\ref{sec:notation-discrete-tree}. Moreover, the root $\mathbf{u}_n$ of $[\widecheck{\mathsf{T}}]^{n}$ corresponds to the root $\varnothing$ of $[\widehat{\mathsf{T}}]^{n}$, and the vertex $\mathbf{u}_0$ in $[\widecheck{\mathsf{T}}]^{n}$ corresponds to the vertex $\mathbf{v}_n$ in $[\widehat{\mathsf{T}}]^{n}$.

\section{Proof of Theorem~\ref{thm:main-dis}}
\label{sec:proof-thm1}

Before we start the proof of Theorem~\ref{thm:main-dis}, let us emphasize that Theorem~\ref{thm:main-dis} is not a straightforward consequence of Theorem~\ref{thm:main-continu}. In fact, we will not directly use the results \eqref{eq:loc-dim-harm} and \eqref{eq:loc-dim-unif} obtained in the continuous setting. 

Recall that $\lambda=\E[\widehat{\mathcal{C}}]-1$ is the constant greater than 1 that appears in the convergence~(\ref{eq:loc-dim-harm}). Let $\delta>0$. By applying Proposition~\ref{prop:size-bias-palm} to the indicator function
\begin{displaymath}
F(\mathcal{T},v)=\mathbf{1}\{n^{-\lambda-\delta}\leq P^{\mathcal{T}}(X_{\tau_n}=v)\leq n^{-\lambda+\delta}\}^{c}, \qquad \mbox{for }\mathcal{T}\in \mathbb{T}_f \mbox{ and } v\in \mathcal{T},
\end{displaymath}
we see that
\begin{displaymath}
\E \Bigg[\sum\limits_{v\in \mathsf{T}_n^{(n)}}^{} \!\! \mathbf{1} \{n^{-\lambda-\delta}\leq P^{\mathsf{T}^{(n)}}(X_{\tau_n}=v)\leq n^{-\lambda+\delta}\}^{c} \Bigg] = \frac{1}{q_n} \P\Big( \{n^{-\lambda-\delta}\leq P^{[\widehat{\mathsf{T}}]^{(n)}}(X_{\tau_n}=\mathbf{v}_n)\leq n^{-\lambda+\delta}\}^{c} \Big).
\end{displaymath}
Notice that in the left-hand side of the last display, $P^{\mathsf{T}^{(n)}}(X_{\tau_n}=v)$ is by definition the harmonic measure $\mu_n(v)$ at vertex $v$. By virtue of (\ref{eq:11}) and (\ref{eq:kolmogorov}), the proof of convergence (\ref{eq:thm-main}) is thus reduced to showing that for every $\delta>0$,
\begin{equation}
\label{eq:2}
\lim_{n\to \infty}  \P\Big(n^{-\lambda-\delta}\leq P^{[\widehat{\mathsf{T}}]^{(n)}}(X_{\tau_n}=\mathbf{v}_n)\leq n^{-\lambda+\delta} \Big) =1.
\end{equation}

Since the hitting distribution of generation $n$ is the same for simple random walk on $[\widehat{\mathsf{T}}]^{(n)}$ and on its reduced tree $[\widehat{\mathsf{T}}]^{*n}$, we have the equality
\begin{equation*}
P^{[\widehat{\mathsf{T}}]^{(n)}}(X_{\tau_n}=\mathbf{v}_n) =P^{[\widehat{\mathsf{T}}]^{*n}}(X_{\tau_n}=\mathbf{v}_n)=P^{[\widehat{\mathsf{T}}]^{n}}(X_{\tau_n}=\mathbf{v}_n)
\end{equation*}
under the probability measure $\P$. Furthermore, according to the final remark in Section~\ref{sec:backward-size-biased}, the (random) probability $P^{[\widehat{\mathsf{T}}]^{n}}(X_{\tau_n}=\mathbf{v}_n)$ is distributed under $\P$ as
$$P^{[\check{\mathsf{T}}]^n}(X_{\tau_n}=\mathbf{u}_0),$$
by which we mean the probability that a simple random walk on $[\widecheck{\mathsf{T}}]^n$ starting from the root $\mathbf{u}_n$ hits level $n$ (of $[\widecheck{\mathsf{T}}]^n$) for the first time at $\mathbf{u}_0$. So the convergence (\ref{eq:2}) is equivalent to
\begin{equation}
\label{eq:3}
\lim_{n\to \infty}  \P\Big(n^{-\lambda-\delta}\leq P^{[\check{\mathsf{T}}]^n}(X_{\tau_n}=\mathbf{u}_0)\leq n^{-\lambda+\delta} \Big) =1.
\end{equation}

In order to show the latter convergence, we denote by $-M_1,-M_2,\ldots$ the generations of the vertices on the spine of $\widecheck{\mathsf{T}}$ where there is (at least) one grafted plane tree that reaches generation 0, i.e.~has a descendant of generation 0. This sequence of negative integers $(-M_k)_{k\geq 1}$ is listed in the strict decreasing order, and we set by convention $M_0=0$. For every $k\geq 1$, we also set $L_k \colonequals M_k-M_{k-1}\geq 1$.

\smallskip
For every $n\geq 1$, let $k_n \colonequals k_n(\widecheck{\mathsf{T}})$ be the index such that $M_{k_n}\leq n< M_{k_n+1}$.
\begin{lemma}
\label{lem:kn-asymptotic}
We have $\P$-a.s.
\begin{displaymath}
\lim_{n\to\infty} \frac{k_n}{2\log n} = 1.
\end{displaymath}
\end{lemma}
\begin{proof}
Recall that for every $j\geq 1$, there are $\widehat N_j-1$ independent Galton--Watson trees grafted at $\mathbf{u}_j$ in $\widecheck{\mathsf{T}}$. Consider the event that at least one of those plane trees grafted at $\mathbf{u}_j$ reaches generation 0, and let $\epsilon_j$ be the corresponding indicator function. Then,
\begin{displaymath}
\P(\epsilon_j=0)=\E\Big[(1-q_{j-1})^{\widehat N_j-1}\Big]=\E\Big[(1-q_{j-1})^{\widehat N-1}\Big].
\end{displaymath}
Let $g_{\theta}$ be the generating function of $\theta$, i.e.
\begin{displaymath}
g_{\theta}(r) \colonequals \sum\limits_{k\geq 0}^{} \theta(k)r^k\,, \qquad 0\leq r\leq 1.
\end{displaymath}
Since $\theta$ has a finite variance $\sigma^2$,
\begin{displaymath}
g_{\theta}(1-s)=1-s+\frac{\sigma^2}{2}s^2 + o(s^2) \qquad \mbox{ as } s\to 0.
\end{displaymath}
As the mean of $\widehat N-1$ is $\sigma^2$, we have $\E\big[(1-s)^{\widehat N-1}\big]=1-\sigma^2 s + o(s)$ as $s\to 0$, which, together with (\ref{eq:kolmogorov}), yields that
\begin{equation}
\label{eq:epsilon-estimate}
\P(\epsilon_j=0)= 1-\frac{2}{j}+o(j^{-1}) \qquad \mbox{ as } j\to \infty.
\end{equation}
Notice that by definition, $k_n=\epsilon_1+\epsilon_2+\cdots+\epsilon_n$. Hence,
\begin{displaymath}
\E[k_n]=\sum\limits_{j=1}^n \big(1-\P(\epsilon_j=0)\big)\, \sim \, 2\log n \qquad \mbox{as } n\to \infty.
\end{displaymath}
Since $\epsilon_1, \ldots,\epsilon_n$ are independent, we also have $\mathrm{var}(k_n)=O(\log n)$, and the $L^2$-convergence of $k_n/(2\log n)$ follows immediately. The a.s.~convergence is then obtained by standard monotonicity and Borel--Cantelli arguments.
\end{proof}

We introduce some additional notation before stating the next proposition. For every $j \geq 0$, we write $P^{\check{\mathsf{T}}}_j$ for the (quenched) probability measure under which we consider a simple random walk $X=(X_k)_{k\geq 0}$ on $\widecheck{\mathsf{T}}$ starting from the vertex $\mathbf{u}_j$. Under $P^{\check{\mathsf{T}}}_j$, we denote by $S_0$ the hitting time of generation 0 by the simple random walk $X$, and for every $i \geq 0$, $\Pi_i \colonequals \inf\{k\geq 0\colon X_k=\mathbf{u}_i\}$ denotes the hitting time of vertex $\mathbf{u}_i$.

\begin{proposition}
\label{prop:red-to-backward}
For every $\delta>0$, there exists an integer $n_0\in \N$ such that for every $n\geq n_0$, we have
\begin{displaymath}
\P\Big(P^{[\check{\mathsf{T}}]^n}(X_{\tau_n}=\mathbf{u}_0)\geq n^{-\lambda+\delta}\Big) \leq 8\, \P\Big(P^{\check{\mathsf{T}}}_{M_{k_n}}(X_{S_0}=\mathbf{u}_0, S_0<\Pi_{M_{k_n+1}})\geq n^{-\lambda+\delta}/2 \Big),
\end{displaymath}
and
\begin{displaymath}
\P\Big(P^{[\check{\mathsf{T}}]^n}(X_{\tau_n}=\mathbf{u}_0)\leq n^{-\lambda-\delta}\Big) \leq 8\, \P\Big(P^{\check{\mathsf{T}}}_{M_{k_n}}(X_{S_0}=\mathbf{u}_0, S_0<\Pi_{M_{k_n+1}})\leq n^{-\lambda-\delta} \Big).
\end{displaymath}
\end{proposition}

\begin{proof}
We keep the notation used in the proof of Lemma~\ref{lem:kn-asymptotic}. Observe that
\begin{displaymath}
\P(M_{k_n+1}-n> n ) = \P(\epsilon_{n+j}=0 \mbox{ for all } 1\leq j\leq n) =\prod_{j=1}^n \P(\epsilon_{n+j}=0) 
\end{displaymath}
converges to $1/4$ as $n\to \infty$ by (\ref{eq:epsilon-estimate}). One can thus find an integer $n_0$ such that for every $n\geq n_0$,
\begin{equation}
\label{eq:1/8-inequality}
\P(M_{k_n+1}>2n )\geq \frac{1}{8}.
\end{equation}
Since $M_{k_n+1}-n$ is independent of the finite tree above the vertex $\mathbf{u}_n$ in $\widecheck{\mathsf{T}}$,
\begin{equation}
\label{eq:prod-ind}
\P\big(M_{k_n+1}>2n, P^{[\check{\mathsf{T}}]^{n}}(X_{\tau_n}=\mathbf{u}_0)\geq n^{-\lambda+\delta}\big) =\P(M_{k_n+1}>2n)\times \P\big(P^{[\check{\mathsf{T}}]^n}(X_{\tau_n}=\mathbf{u}_0)\geq n^{-\lambda+\delta}\big).
\end{equation}

On the other hand, it is crucial to notice that under the probability measure $\P$, the probability $P^{[\check{\mathsf{T}}]^n}(X_{\tau_n}=\mathbf{u}_0)$ has the same distribution as the conditional probability
$$P^{\check{\mathsf{T}}}_n(X_{S_0}=\mathbf{u}_0\!\mid\! S_0< \Pi_{M_{k_n+1}}),$$
which can be calculated as
\begin{eqnarray*}
P^{\check{\mathsf{T}}}_n(X_{S_0}=\mathbf{u}_0\!\mid\! S_0< \Pi_{M_{k_n+1}})\! & =& \!\frac{P^{\check{\mathsf{T}}}_n(X_{S_0}=\mathbf{u}_0, S_0< \Pi_{M_{k_n+1}})}{P^{\check{\mathsf{T}}}_n(S_0< \Pi_{M_{k_n+1}})}\\
&=& \!\frac{P^{\check{\mathsf{T}}}_n(\Pi_{M_{k_n}}<\Pi_{M_{k_n+1}})\times P^{\check{\mathsf{T}}}_{M_{k_n}}\!(X_{S_0}=\mathbf{u}_0, S_0< \Pi_{M_{k_n+1}})}{P^{\check{\mathsf{T}}}_n(S_0< \Pi_{M_{k_n+1}})}
\end{eqnarray*}
by the strong Markov property of the random walk. Besides, simple considerations show that $P^{\check{\mathsf{T}}}_n(S_0< \Pi_{M_{k_n+1}})\geq 1/2$ on the event $\{M_{k_n+1}>2n\}$. Hence, we have
\begin{align*}
\P\big(M_{k_n+1}\!>2n, P^{[\check{\mathsf{T}}]^n}(X_{\tau_n}=\mathbf{u}_0)\geq n^{-\lambda+\delta}\big) &= \P\big(M_{k_n+1}\!>2n, P^{\check{\mathsf{T}}}_n(X_{S_0}=\mathbf{u}_0\!\mid\! S_0< \Pi_{M_{k_n+1}})\geq  n^{-\lambda+\delta} \big) \\
&\leq  \P\big(M_{k_n+1}\!>2n, P^{\check{\mathsf{T}}}_{M_{k_n}}\!(X_{S_0}=\mathbf{u}_0, S_0< \Pi_{M_{k_n+1}})\geq n^{-\lambda+\delta}/2\big)\\
&\leq  \P\big( P^{\check{\mathsf{T}}}_{M_{k_n}}\!(X_{S_0}=\mathbf{u}_0, S_0< \Pi_{M_{k_n+1}})\geq n^{-\lambda+\delta}/2\big).
\end{align*}
The last display, together with (\ref{eq:1/8-inequality}) and (\ref{eq:prod-ind}), yields the first inequality in the statement of the proposition.

We can argue in a similar manner for the second inequality stated in the proposition. Its proof is even simpler because it suffices to use the bound
\begin{displaymath}
P^{\check{\mathsf{T}}}_{M_{k_n}}(X_{S_0}=\mathbf{u}_0, S_0< \Pi_{M_{k_n+1}}) \leq P^{\check{\mathsf{T}}}_n(X_{S_0}=\mathbf{u}_0\!\mid\! S_0< \Pi_{M_{k_n+1}}),
\end{displaymath}
instead of the estimate $P^{\check{\mathsf{T}}}_n(S_0< \Pi_{M_{k_n+1}})\geq 1/2$ used above.
\end{proof}

According to (\ref{eq:3}) and the preceding result, we can therefore derive Theorem~\ref{thm:main-dis} from the following proposition.
\begin{proposition}
\label{prop:red-theorem}
For every $\delta>0$, it holds that
\begin{equation}
\label{eq:red-theorem}
\lim_{n\to \infty}  \P\Big(n^{-\lambda-\delta}\leq P^{\check{\mathsf{T}}}_{M_{k_n}}\!(X_{S_0}=\mathbf{u}_0, S_0<\Pi_{M_{k_n+1}}) \leq n^{-\lambda+\delta}  \Big) =1.
\end{equation}
\end{proposition}

\subsection{Proof of Proposition~\ref{prop:red-theorem}}

Under the probability measure $\P$, we set, for every $k\geq 1$, 
\begin{displaymath}
p_k=p_k(\widecheck{\mathsf{T}}) \colonequals P^{\check{\mathsf{T}}}_{M_k}(X_{S_0}=\mathbf{u}_0, S_0<\Pi_{M_{k+1}}).
\end{displaymath}
By the definition of $M_k$, there exists at least one plane tree grafted to $\mathbf{u}_{M_k}$ that reaches generation~0. The root of this subtree is necessarily a child of $\mathbf{u}_{M_k}$ distinct from $\mathbf{u}_{M_k-1}$. If such a subtree is unique, we let $c_k=c_k(\widecheck{\mathsf{T}})$ be the probability that a simple random walk starting from its root reaches generation 0 before visiting $\mathbf{u}_{M_k}$. If there is more than one such grafted trees, we take $c_k$ to be the sum of the corresponding probabilities. This definition is justified by the fact that $c_k$ can be interpreted as the effective conductance between $\mathbf{u}_{M_k}$ and generation 0 in the graph that consists only of the vertex $\mathbf{u}_{M_k}$ and all the subtrees grafted to it.

We also set, for every $k\geq 1$,
\begin{displaymath}
h_k=h_k(\widecheck{\mathsf{T}}) \colonequals P^{\check{\mathsf{T}}}_{M_k-1}(S_0<\Pi_{M_k}),
\end{displaymath}
which is the probability that a simple random walk starting from $\mathbf{u}_{M_k-1}$ reaches generation 0 before visiting $\mathbf{u}_{M_k}$. With the notation of Section~\ref{sec:cv-dis-cond}, it is clear that
\begin{displaymath}
h_k=\mathcal{C}_{M_k-1}([\widecheck{\mathsf{T}}]^{M_k-1}).
\end{displaymath}

We write $\ell_k=1/L_k= (M_k-M_{k-1})^{-1}$ for all $k\geq 1$. Then simple considerations show that
\begin{displaymath}
p_1=\frac{\ell_1}{\ell_1+c_1+\ell_2},
\end{displaymath}
and, for all $k\geq 2$,
\begin{equation}
\label{eq:prob-iter}
p_k=\frac{\ell_k}{\ell_k+c_k+\ell_{k+1}}\Big(p_{k-1}+\frac{\ell_k}{\ell_k+c_{k-1}+h_{k-1}}\,p_k  \Big).
\end{equation}
To establish the last formula, we consider the excursions of simple random walk outside of vertex $\mathbf{u}_{M_k}$, which are independent of the same law. Under this excursion law, the random walk makes its first jump with equal probability towards one of its neighbors, which are $\mathbf{u}_{M_k-1}$, $\mathbf{u}_{M_k+1}$ and the children of $\mathbf{u}_{M_k}$ distinct from $\mathbf{u}_{M_k-1}$. The respective probabilities for an excursion to visit $\mathbf{u}_{M_{k-1}}$, to visit $\mathbf{u}_{M_{k+1}}$, and to reach generation 0 in one of the subtrees grafted at $\mathbf{u}_{M_k}$, are proportional respectively to $1/L_k=\ell_k$, to $1/L_{k+1}=\ell_{k+1}$ and to $c_k$. So the probability for the random walk starting from $\mathbf{u}_{M_k}$ to visit $\mathbf{u}_{M_{k-1}}$ before hitting $\mathbf{u}_{M_{k+1}}$ or reaching generation 0 is
\begin{displaymath}
\frac{\ell_k}{\ell_k+\ell_{k+1}+c_k}.
\end{displaymath}
Next, conditionally on the latter event, the strong Markov property leads us to consider a simple random walk that starts from $\mathbf{u}_{M_{k-1}}$. With probability $p_{k-1}$ it reaches generation 0 by hitting the vertex $\mathbf{u}_0$ before moving down to $\mathbf{u}_{M_k}$. However, we must also add the probability that this random walk goes back down to $\mathbf{u}_{M_k}$ before reaching generation 0, which is equal to
\begin{displaymath}
\frac{\ell_k}{\ell_k+c_{k-1}+h_{k-1}},
\end{displaymath}
multiplied by the probability $p_k$ that once returning to $\mathbf{u}_{M_k}$ the random walk will eventually hit generation 0 at $\mathbf{u}_0$ before moving down to $\mathbf{u}_{M_{k+1}}$.

We derive from~(\ref{eq:prob-iter}) that
\begin{displaymath}
p_{k-1}= p_k\Big( \frac{\ell_k+c_k+\ell_{k+1}}{\ell_k}-\frac{\ell_k}{\ell_k+c_{k-1}+h_{k-1}}\Big),
\end{displaymath}
from which it follows that
\begin{displaymath}
p_1=p_k \times \prod_{j=2}^k\Big( 1+ \frac{c_j+\ell_{j+1}}{\ell_j}-\frac{\ell_j}{\ell_j+c_{j-1}+h_{j-1}}\Big).
\end{displaymath}
We define thus, for every $j\geq 2$,
\begin{displaymath}
Q_j=Q_j(\widecheck{\mathsf{T}}) \colonequals  \log \Big( 1+\frac{c_j+\ell_{j+1}}{\ell_j}-\frac{\ell_j}{\ell_j+c_{j-1}+h_{j-1}}\Big).
\end{displaymath}

\begin{lemma}
\label{lem:sum-L2-conv}
We have
\begin{equation}
\label{eq:sum-L2-conv}
\frac{1}{k} \sum\limits_{j=2}^k Q_j  \,\xrightarrow[k\to\infty]{L^2(\P)} \, \frac{\lambda}{2}.
\end{equation}
\end{lemma}

The proof of this key lemma is postponed to the next section. Let us first show how it implies Proposition~\ref{prop:red-theorem} and thus Theorem~\ref{thm:main-dis}. For any $\delta>0$, consider the event
\begin{displaymath}
\Big\{(\lambda-\delta)\log n\leq \sum\limits_{j=2}^{k_n} Q_j \leq (\lambda+\delta)\log n \Big\}.
\end{displaymath}
Using Lemma~\ref{lem:kn-asymptotic} and Lemma~\ref{lem:sum-L2-conv}, we see that the last event holds with $\P$-probability tending to~1 as $n\to \infty$. As
\begin{displaymath}
p_{k_n}=p_1 \exp\Big(-\sum\limits_{j=2}^{k_n} Q_j \Big),
\end{displaymath}
we have
\begin{displaymath}
\lim_{n\to \infty}  \P\Big( p_1 n^{-\lambda-\delta}\leq p_{k_n}\leq p_1 n^{-\lambda+\delta}  \Big) =1.
\end{displaymath}
Recalling the definition of $p_k$, we conclude that
\begin{displaymath}
\lim_{n\to \infty}  \P\Big( p_1 n^{-\lambda-\delta}\leq P^{\check{\mathsf{T}}}_{M_{k_n}}(X_{S_0}=\mathbf{u}_0, S_0<\Pi_{M_{k_n+1}}) \leq p_1 n^{-\lambda+\delta}  \Big) =1.
\end{displaymath}
Since $\delta$ is arbitrary, the required convergence (\ref{eq:red-theorem}) readily follows from the last display. Therefore, it remains to prove Lemma~\ref{lem:sum-L2-conv}.

\subsection{Proof of Lemma~\ref{lem:sum-L2-conv}}

By the definition of $Q_k$, we can write for every $k\geq 2$,
\begin{displaymath}
Q_k=\log \bigg(1+ \frac{M_kc_k+\frac{M_k}{L_{k+1}}}{\frac{M_k}{L_k}}- \frac{\frac{M_k}{L_k}}{\frac{M_k}{L_k}+\frac{M_k}{M_{k-1}}(M_{k-1}c_{k-1}+M_{k-1}h_{k-1})}\bigg).
\end{displaymath}

\begin{lemma}
\label{lemma:5rv-cv-in-law}
We have
\begin{displaymath}
\Big(\frac{L_{k+1}}{M_k},\frac{L_k}{M_{k-1}},M_kc_k,M_{k-1}c_{k-1},M_{k-1}h_{k-1}\Big)  \, \build{\longrightarrow}_{k\to\infty}^{(\mathrm{d})}\, \big(\mathcal{R},\mathcal{R}',\mathcal{C},\mathcal{C}', \widehat{\mathcal{C}}\, \big),
\end{displaymath}
where in the limit
\begin{itemize}
\item $\mathcal{R}$ and $\mathcal{R}'$ are two positive random variables with the same distribution given by
    $$\P(\mathcal{R}>x)=(1+x)^{-2}  \mbox{ for all } x\geq 0;$$
\item $\mathcal{C}$ and $\mathcal{C}'$ are distributed according to the law $\gamma$;
\item $\widehat{\mathcal{C}}$ is distributed according to the law $\widehat \gamma$.
\end{itemize}
Furthermore, we suppose that $\mathcal{R},\mathcal{R}',\mathcal{C},\mathcal{C}', \widehat{\mathcal{C}}$ are all defined under the probability measure $\P$, and they are independent.
\end{lemma}

\begin{proof}
We first observe that $(M_{k-1},L_k)_{k\geq 1}$ is a homogeneous Markov chain on $\mathbb{Z}_+\times \mathbb{N}$ whose initial distribution and transition probabilities are given as follows. Initially $M_0=0$ and for every integer $\ell\geq 1$,
\begin{displaymath}
\P(L_1>\ell) = \prod_{j=1}^{\ell} \P(\epsilon_j=0),
\end{displaymath}
where $\{\epsilon_j=0\}$ means as previously that none of the Galton--Watson trees grafted at $\mathbf{u}_j$ reaches generation 0. Then for every $k\geq 1$, $M_{k}=M_{k-1}+L_{k}$ and conditionally on $\{M_{k}=m\}$,
\begin{equation}
\label{eq:LMproduct}
\P(L_{k+1}>\ell \mid M_{k}=m)= \prod_{j=1}^{\ell} \P(\epsilon_{m+j}=0)\equalscolon F(m,\ell) , \qquad \mbox{for every }\ell \geq 1.
\end{equation}
Using~(\ref{eq:epsilon-estimate}), it is elementary to verify that for every $x>0$,
\begin{displaymath}
F(m,\lfloor xm\rfloor) \build{\longrightarrow}_{m \to\infty}^{} \frac{1}{(1+x)^2}.
\end{displaymath}
For every $k\geq 1$, let $\mathcal{F}_k=\sigma(L_1,L_2,\ldots,L_k)$ be the $\sigma$-field generated by $(L_i, 1\leq i\leq k)$, so that $(\mathcal{F}_k)_{k\geq 1}$ is the natural filtration associated with the Markov chain $(M_{k-1},L_k)_{k\geq 1}$. As $M_{k}\geq k$, it is clear that $M_{k}\to \infty$ as $k\to \infty$. By dominated convergence and the last display, we get that for every $x,y>0$,
\begin{eqnarray*}
\P\Big(\frac{L_k}{M_{k-1}}>x, \frac{L_{k+1}}{M_k}>y\Big) & \!=\!& \E\bigg[ \mathbf{1}\Big\{\frac{L_k}{M_{k-1}}>x\Big\}\, \P\left(\frac{L_{k+1}}{M_k}>y \,\middle|\, \mathcal{F}_k\right)\bigg] \\
& \!=\!& \E\bigg[\mathbf{1}\Big\{\frac{L_k}{M_{k-1}}>x\Big\} \,\P\left(\frac{L_{k+1}}{M_k}>y \,\middle|\, M_k\right)\bigg] \\
& \!\!\build{\longrightarrow}_{k\to\infty}^{}\!\! &  \frac{1}{(1+x)^2(1+y)^2},
\end{eqnarray*}
which entails that
\begin{equation}
\label{eq:ML-conv}
\Big(\frac{L_k}{M_{k-1}},\frac{L_{k+1}}{M_k}\Big)  \, \build{\longrightarrow}_{k\to\infty}^{(\mathrm{d})}\, \big(\mathcal{R}',\mathcal{R}\big).
\end{equation}

Note that conditionally on $M_k$ and on the number of subtrees grafted at $\mathbf{u}_{M_k}$ that hit generation 0, these subtrees are Galton--Watson trees conditioned to have height greater than $M_k-1$. Furthermore, the property $\E[\widehat N]=\sum k^2\theta(k)<\infty$ and the estimate~(\ref{eq:kolmogorov}) entail that a.s.~for all sufficiently large $k$, there is a unique subtree grafted at $\mathbf{u}_{M_k}$ that reaches generation 0. Hence by Proposition~\ref{prop:cv-conductance}, we obtain the convergence
\begin{equation}
\label{eq:c-conv}
\big(M_{k-1}c_{k-1},M_kc_k\big)  \, \build{\longrightarrow}_{k\to\infty}^{(\mathrm{d})}\, \big(\mathcal{C}',\mathcal{C}\big),
\end{equation}
which holds jointly with~(\ref{eq:ML-conv}), provided we let $(\mathcal{C}',\mathcal{C})$ be independent of $(\mathcal{R}',\mathcal{R})$.

Let $J\geq 2$ be a fixed integer. We can generalize the preceding arguments to show that the $2J$-tuple
$$\Big(\frac{L_{k-1}}{M_{k-2}}, \frac{L_{k-2}}{M_{k-3}}, \ldots, \frac{L_{k-J}}{M_{k-J-1}}, M_{k-2}c_{k-2}, M_{k-3}c_{k-3},\ldots, M_{k-J-1}c_{k-J-1}\Big)$$
converges in distribution as $k\to \infty$ to $(\mathcal{R}_1,\mathcal{R}_2,\ldots,\mathcal{R}_J,\mathcal{C}_1,\mathcal{C}_2,\ldots,\mathcal{C}_J)$. These random variables appearing in the limit are all independent, and $(\mathcal{R}_j)_{1\leq j\leq J}$, respectively $(\mathcal{C}_j)_{1\leq j\leq J}$, have the same distribution as $\mathcal{R}$, resp.~as $\mathcal{C}$. If we set $V_j=\frac{\mathcal{R}_j}{1+\mathcal{R}_j}$ for all $1\leq j\leq J$, then $(V_j)_{1\leq j\leq J}$ are i.i.d.~with the same law of density $2(1-x)$ on $[0,1]$. Moreover, the previous convergence can be reformulated as
\begin{equation}
\label{eq:2J-cv-in-law}
\Big(\frac{L_{k-j}}{M_{k-j}},M_{k-j-1}c_{k-j-1}\Big)_{1\leq j\leq J}  \, \build{\longrightarrow}_{k\to\infty}^{(\mathrm{d})}\, (V_j, \mathcal{C}_j)_{1\leq j\leq J}.
\end{equation}
For all integers $k>J$ and $0\leq j \leq J$, we can define
\begin{displaymath}
h_k^{(j)}=h_k^{(j)}(\widecheck{\mathsf{T}}) \colonequals P^{\check{\mathsf{T}}}_{M_k-1}(S_0\wedge \Pi_{M_{k-j}-1}<\Pi_{M_k}),
\end{displaymath}
which is the probability that a simple random walk starting from $\mathbf{u}_{M_k-1}$ reaches generation 0 or the vertex $\mathbf{u}_{M_{k-j}-1}$ before visiting $\mathbf{u}_{M_k}$. By definition it is clear that $h_k^{(0)}=1$. From the interpretation of $h_{k-1}^{(J)}$ as an electric conductance, we obtain the formula
$$h_{k-1}^{(J)}=\Big( L_{k-1}+\big(c_{k-2}+h_{k-2}^{(J-1)}\big)^{-1}\Big)^{-1}$$
by the series law and parallel law. It follows that
$$M_{k-1}h_{k-1}^{(J)}= \left( \frac{L_{k-1}}{M_{k-1}}+\frac{1-\frac{L_{k-1}}{M_{k-1}}}{M_{k-2}c_{k-2}+M_{k-2}h_{k-2}^{(J-1)}}\right)^{-1}.$$
The same calculation can be repeated for $M_{k-2}h_{k-2}^{(J-1)},M_{k-3}h_{k-3}^{(J-2)} \ldots$ up until $M_{k-J}h_{k-J}^{(1)}$. By using~(\ref{eq:2J-cv-in-law}) and the fact that $h_{k-J-1}^{(0)}=1$, we see that the law of $M_{k-1}h_{k-1}^{(J)}$ converges weakly to the $J$-th iterate $\widehat \Phi^J(\delta_\infty)$ as $k\to \infty$, where $\widehat \Phi$ is the mapping defined in~(\ref{eq:iteration-map}). Moreover, according to assertion~(1) in Proposition~\ref{prop:c*-law}, $\widehat \Phi^J(\delta_\infty)$ converges weakly to $\widehat \gamma=\mathsf{Law}(\widehat{\mathcal{C}})$ as $J\to \infty$.

On the other hand,
$$ 0\leq h_{k}^{(J)}-h_{k} = P^{\check{\mathsf{T}}}_{M_k-1}(\Pi_{M_{k-J}-1}<\Pi_{M_k}<S_0).$$
Note that there is probability at least $1-\frac{M_{k-J}-1}{M_k}$ that, after hitting the vertex $\mathbf{u}_{M_{k-J}-1}$, the simple random walk on $\widecheck{\mathsf{T}}$ will reach generation 0 before hitting $\mathbf{u}_{M_k}$. Hence, by the strong Markov property of simple random walk, we have
$$h_{k}^{(J)}-h_{k} \leq P^{\check{\mathsf{T}}}_{M_k-1}(\Pi_{M_{k-J}-1}<S_0\wedge \Pi_{M_k})\times \frac{M_{k-J}-1}{M_k}. $$
By similar reasoning,
$$P^{\check{\mathsf{T}}}_{M_k-1}(\Pi_{M_{k-J}-1}<S_0\wedge \Pi_{M_k}) \leq \frac{1}{M_k-M_{k-J}+1},$$
and it follows that
$$M_kh_{k}^{(J)}-M_kh_{k}\leq \frac{M_{k-J}-1}{M_k-M_{k-J}+1}.$$
Thus, for any $\eta>0$,
$$\P \big(|M_{k-1}h_{k-1}-M_{k-1}h_{k-1}^{(J)}|\geq\eta \big) \leq \P\Big( \frac{M_{k-J-1}-1}{M_{k-1}-M_{k-J-1}+1} \geq \eta\Big). $$
But due to the previous discussions, it is clear that
$$ \lim_{J\to \infty} \limsup_{k\to \infty}\P\Big(\frac{M_{k-J-1}-1}{M_{k-1}-M_{k-J-1}+1}\geq \eta\Big) =0.$$
So we obtain, for any $\eta>0$, that
$$ \lim_{J\to \infty} \limsup_{k\to \infty} \P \big(|M_{k-1}h_{k-1}-M_{k-1}h_{k-1}^{(J)}|\geq\eta \big) =0.$$
Finally, by applying \cite[Theorem 3.2]{Bill99} we get that
\begin{equation}
\label{eq:h-cv-in-law}
M_{k-1}h_{k-1} \, \build{\longrightarrow}_{k\to\infty}^{(\mathrm{d})}\, \widehat{\mathcal{C}}.
\end{equation}
Notice that $h_{k-1}$ only depends on $M_{k-1}$ and the finite tree strictly above the vertex $\mathbf{u}_{M_{k-1}}$ in~$\widecheck{\mathsf{T}}$. So conditionally on $M_{k-1}$, the latter quantity $h_{k-1}$ is independent of $(L_k,L_{k+1},c_{k-1},c_k)$. In consequence, the convergence (\ref{eq:h-cv-in-law}) holds jointly with (\ref{eq:ML-conv}) and (\ref{eq:c-conv}), if we take $\widehat{\mathcal{C}}$ to be independent of $(\mathcal{R},\mathcal{R}',\mathcal{C},\mathcal{C}')$. The proof of Lemma~\ref{lemma:5rv-cv-in-law} is therefore complete.
\end{proof}

\begin{lemma}
\label{lemma:L2-bdd-hk}
It holds that
$$\sup_{k\geq 1} \E\big[(M_kh_k)^2\big]<\infty.$$
\end{lemma}

\begin{proof}
From the interpretation of $h_k$ as a conductance, we know by the series law and parallel law that
$$h_k=\Big(L_k+\big(c_{k-1}+h_{k-1}\big)^{-1}\Big)^{-1}.$$
It follows that
$$M_kh_k=\frac{M_{k-1}c_{k-1}+M_{k-1}h_{k-1}}{\frac{L_k}{M_k}(M_{k-1}c_{k-1}+M_{k-1}h_{k-1})+1-\frac{L_k}{M_k}}\,.$$
Again from the interpretation of $c_{k-1}$ and $h_{k-1}$ as conductances, it is elementary to see that $M_{k-1}c_{k-1}\geq 1$ and $M_{k-1}h_{k-1}\geq 1$. Hence,
$$M_k h_k \,\leq\, \frac{M_{k-1}c_{k-1}+M_{k-1}h_{k-1}}{\frac{2L_k}{M_k}+1-\frac{L_k}{M_k}} \,\leq\, M_{k-1}c_{k-1}+\frac{M_{k-1}h_{k-1}}{1+\frac{L_k}{M_k}}.$$
For any $\eta>0$, one can take $C(\eta)>1+\frac{1}{\eta}$ so that $(a+b)^2\leq C(\eta)a^2+(1+\eta)b^2$ for every $a,b>0$. Applying this inequality to the last display, we obtain
\begin{equation}
\label{eq:L2-hk-majoration}
\E\big[(M_kh_k)^2\big] \leq C(\eta)\E\big[(M_{k-1}c_{k-1})^2\big]+(1+\eta)\E\bigg[\Big(\frac{M_{k-1}h_{k-1}}{1+\frac{L_k}{M_k}}\Big)^2 \bigg].
\end{equation}

Now notice that
$$\E\bigg[\Big(\frac{M_{k-1}h_{k-1}}{1+\frac{L_k}{M_k}}\Big)^2 \bigg] = \E\bigg[\sum\limits_{\ell=1}^{\infty}\P(L_k=\ell \!\mid\! M_{k-1})\frac{(M_{k-1}h_{k-1})^2}{(1+\frac{\ell}{M_{k-1}+\ell})^2} \bigg].$$
According to (\ref{eq:LMproduct}), there exists a constant $c>0$ such that, a.s.~for all integers $k\geq 2$, we have $\P(L_k\geq M_{k-1}\!\mid\! M_{k-1})\geq c$. It follows that
\begin{align*}
& \; \sum\limits_{\ell \geq 1}^{} \P(L_k=\ell \!\mid\! M_{k-1})\frac{(M_{k-1}h_{k-1})^2}{(1+\frac{\ell}{M_{k-1}+\ell})^2}\\
\leq \; & \; \big(\frac{2}{3}\big)^2 \sum\limits_{\ell\geq M_{k-1}}^{} \P(L_k=\ell \!\mid\! M_{k-1}) (M_{k-1}h_{k-1})^2 +\!\sum\limits_{\ell <M_{k-1}}^{} \P(L_k=\ell \!\mid\! M_{k-1})(M_{k-1}h_{k-1})^2 \\
\leq \;& \; \big(1-\frac{5}{9}c\big)(M_{k-1}h_{k-1})^2,
\end{align*}
and thus
$$\E\bigg[\Big(\frac{M_{k-1}h_{k-1}}{1+\frac{L_k}{M_k}}\Big)^2 \bigg] \leq \big(1-\frac{5}{9}c\big) \E\big[(M_{k-1}h_{k-1})^2\big].$$
Together with (\ref{eq:L2-hk-majoration}), the last display entails that
$$\E\big[(M_kh_k)^2\big] \leq  C(\eta)\E\big[(M_{k-1}c_{k-1})^2\big]+(1+\eta)(1-\frac{5}{9}c) \E\big[(M_{k-1}h_{k-1})^2\big].$$

Recall that by Lemma~\ref{lem:L2-gw-condct}, $\E\big[(M_{k-1}c_{k-1})^2\big]$ is uniformly bounded with respect to $k$. So by choosing $\eta$ sufficiently small so that $(1+\eta)(1-\frac{5}{9}c)<1$, we see that there exist positive constants $C<\infty$ and $\rho<1$, both independent of $k$, such that for all $k\geq 2$,
$$\E\big[(M_kh_k)^2\big] \leq C+\rho \,\E\big[(M_{k-1}h_{k-1})^2\big].$$
The sequence $(\E[(M_kh_k)^2])_{k\geq 1}$ is therefore bounded.
\end{proof}

With the notation of Lemma~\ref{lemma:5rv-cv-in-law}, we set
\begin{displaymath}
Q_{\infty}\colonequals \log \bigg( 1+ \frac{\mathcal{C}+\frac{1}{\mathcal{R}}}{\frac{1+\mathcal{R}'}{\mathcal{R}'}}- \frac{1}{1+\mathcal{R}'\big(\mathcal{C}'+\widehat{\mathcal{C}}\,\big)} \bigg).
\end{displaymath}

\begin{lemma}
\label{lemma:auxiliary-L1}
\begin{enumerate}
\item[(i)] We have $\lim_{k\to \infty}\E[Q_k]= \E[Q_{\infty}]$.
\item[(ii)] We have the equality $\E[Q_{\infty}] = \lambda /2$.
\item[(iii)] It holds that
\begin{displaymath}
\sup_{i,j\geq 1} \E\big[|Q_iQ_j|\big]<\infty \quad \mbox{ and } \quad \sup_{i,j\geq 1} \E\big[(Q_iQ_j)^2\big]<\infty.
\end{displaymath}
\end{enumerate}
\end{lemma}

\begin{proof}
(i) On the one hand, since $\frac{M_k}{L_k}>1$,
\begin{displaymath}
Q_k\leq \log \bigg(1+ \frac{M_kc_k+\frac{M_k}{L_{k+1}}}{\frac{M_k}{L_k}}\bigg) \leq \log \Big( 1+M_kc_k+\frac{M_k}{L_{k+1}}\Big).
\end{displaymath}
On the other hand,
\begin{align*}
Q_k \,\geq & \,\,\log \bigg(1- \frac{\frac{M_k}{L_k}}{\frac{M_k}{L_k}+\frac{M_k}{M_{k-1}}(M_{k-1}c_{k-1}+M_{k-1}h_{k-1})}\bigg) \\
=& \,\,-\log \bigg(1+ \frac{\frac{M_k}{L_k}}{\frac{M_k}{M_{k-1}}(M_{k-1}c_{k-1}+M_{k-1}h_{k-1})}\bigg) \,=\,-\log \bigg(1+ \frac{\frac{M_{k-1}}{L_k}}{M_{k-1}c_{k-1}+M_{k-1}h_{k-1}}\bigg).
\end{align*}
Using the facts that $M_{k-1}c_{k-1}\geq 1$ and $M_{k-1}h_{k-1}\geq 1$, we arrive at
\begin{displaymath}
|Q_k|\leq \max\bigg\{\log \Big( 1+M_kc_k+\frac{M_k}{L_{k+1}}\Big),\log \Big(1+ \frac{M_{k-1}}{2L_k}\Big) \bigg\}.
\end{displaymath}
One can find $A>0$ such that $\log(1+x)\leq A+x^{1/2}$ for every $x>0$. It follows that
\begin{equation}
\label{eq:Q-bound}
|Q_k|\leq A+\Big(M_kc_k+\frac{M_k}{L_{k+1}}\Big)^{\frac{1}{2}}+\Big(\frac{M_{k-1}}{L_k}\Big)^{\frac{1}{2}} \leq A+ (M_kc_k)^{\frac{1}{2}}+\Big(\frac{M_k}{L_{k+1}}\Big)^{\frac{1}{2}}+\Big(\frac{M_{k-1}}{L_k}\Big)^{\frac{1}{2}}.
\end{equation}

Recall the convergence in distribution $M_kc_k\overset{(\mathrm{d})}{\to} \mathcal{C}$ shown in Lemma~\ref{lemma:5rv-cv-in-law}. By Lemma~\ref{lem:L2-gw-condct}, it follows that
\begin{equation}
\label{eq:C-mean-cv}
\E[M_kc_k] \, \build{\longrightarrow}_{k\to\infty}^{}\, \E[\mathcal{C}].
\end{equation}
In particular, $\sup_k\E[M_kc_k]<\infty$. Meanwhile, using (\ref{eq:epsilon-estimate}) and~(\ref{eq:LMproduct}), it is not difficult to verify that there exists a positive constant $K$ such that for every $x>0$,
\begin{displaymath}
\sup_{k\geq 1} \, \P\Big( \frac{M_k}{L_{k+1}}>x \Big)\leq \frac{K}{1+x}.
\end{displaymath}
So using the formula
\begin{displaymath}
\E\bigg[\Big(\frac{M_k}{L_{k+1}}\Big)^{\alpha}\bigg]= \int_0^{\infty} \alpha x^{\alpha-1} \,\P\Big( \frac{M_k}{L_{k+1}}>x\Big) \,\mathrm{d}x\,,
\end{displaymath}
we get the existence of a constant $\alpha\in (\frac{1}{2},1)$ such that
\begin{equation}
\label{eq:1/2-unif-bound}
\sup_{k\geq 1} \, \E \bigg[\Big(\frac{M_k}{L_{k+1}}\Big)^{\alpha}\bigg] <\infty.
\end{equation}
Hence, it follows from (\ref{eq:Q-bound}) that $(Q_k)_{k\geq 2}$ is bounded in $L^p$ with some $p>1$. The sequence $(Q_k)_{k\geq 2}$ is thus uniformly integrable. However, according to Lemma~\ref{lemma:5rv-cv-in-law}, $Q_k$ converges in distribution to~$Q_{\infty}$. Therefore, $Q_k\to Q_{\infty}$ in $L^1$ and we have
\begin{displaymath}
\lim_{k\to \infty}\E[Q_k] =\E[Q_{\infty}].
\end{displaymath}

(ii) Recall that $V=\frac{\mathcal{R}}{1+\mathcal{R}}$ and $V'=\frac{\mathcal{R}'}{1+\mathcal{R}'}$ are independent with the same law of density $2(1-x)$ on $[0,1]$. Noting that
\begin{eqnarray*}
\log \bigg( 1+ \frac{\mathcal{C}+\frac{1}{\mathcal{R}}}{\frac{1+\mathcal{R}'}{\mathcal{R}'}}- \frac{1}{1+\mathcal{R}'\big(\mathcal{C}'+\widehat{\mathcal{C}}\,\big)} \bigg) & =& \log \bigg( \Big(\mathcal{C}+\frac{1}{\mathcal{R}} \Big)V'+ \frac{V'(\mathcal{C}'+\widehat{\mathcal{C}})}{1-V'+V'(\mathcal{C}'+\widehat{\mathcal{C}})}\bigg)\\
&=& \log V' +\log\bigg( \mathcal{C}+\frac{1}{\mathcal{R}}+\Big(V'+\frac{1-V'}{\mathcal{C}'+\widehat{\mathcal{C}}}\Big)^{-1}\bigg),
\end{eqnarray*}
we can use the distributional identity (\ref{eq:c*-rde}) to obtain
\begin{displaymath}
\E\big[Q_{\infty}\big]=\E\big[\log V'\big]+\E\Big[\log \big(\mathcal{C}+\frac{1}{\mathcal{R}}+\widehat{\mathcal{C}}\big)\Big].
\end{displaymath}
Since $\mathcal{R}=\frac{V}{1-V}$, it follows that
\begin{displaymath}
\log \big(\mathcal{C}+\frac{1}{\mathcal{R}}+\widehat{\mathcal{C}}\big) =\log (1-V+V(\mathcal{C}+\widehat{\mathcal{C}}))-\log V,
\end{displaymath}
which yields $\E\big[Q_{\infty}\big] = \E\big[\log (1-V+V(\mathcal{C}+\widehat{\mathcal{C}}))\big]$. To complete the proof of assertion (ii), we apply (\ref{eq:c*-rde}) again to see that
\begin{displaymath}
\E\big[\log (1-V+V(\mathcal{C}+\widehat{\mathcal{C}}))\big]=  \E\big[\log (\widehat{\mathcal{C}}+\mathcal{C})\big]- \E\big[\log(\widehat{\mathcal{C}})\big]= \E\Big[\log \Big(\frac{\widehat{\mathcal{C}}+\mathcal{C}}{\widehat{\mathcal{C}}}\Big)\Big],
\end{displaymath}
which is equal to $\lambda/2$ according to (\ref{eq:lambda-log}).

(iii)  There exists a constant $\widetilde A>0$ such that $\log(1+x)\leq \widetilde A+x^{1/4}$ for every $x>0$. It follows then from the same arguments as in the proof of assertion (i) that
\begin{align*}
|Q_iQ_j| \,\leq\, &\, \bigg(\log \Big( 1+M_ic_i+\frac{M_i}{L_{i+1}}\Big)+\log \Big(1+ \frac{M_{i-1}}{2L_i}\Big)\bigg)\bigg(\log \Big( 1+M_jc_j+\frac{M_j}{L_{j+1}}\Big)+\log \Big(1+ \frac{M_{j-1}}{2L_j}\Big)\bigg)\\
\leq\, &\, \bigg( 2\widetilde A +\Big(M_ic_i+\frac{M_i}{L_{i+1}}\Big)^{\frac{1}{4}}+\Big(\frac{M_{i-1}}{L_i}\Big)^{\frac{1}{4}}\bigg) \bigg( 2\widetilde A +\Big(M_jc_j+\frac{M_j}{L_{j+1}}\Big)^{\frac{1}{4}}+\Big(\frac{M_{j-1}}{L_j}\Big)^{\frac{1}{4}}\bigg).
\end{align*}
In order to prove that $\sup_{i,j}\E[|Q_iQ_j|]< \infty$, it is enough to develop the product in the last line of the preceding display, and to show that the expectation of each term is uniformly bounded with respect to $i$ and $j$. In fact, by the Cauchy--Schwarz inequality,
\begin{displaymath}
\E\bigg[\Big(\frac{M_{i-1}}{L_i}\Big)^{\frac{1}{4}}\Big(\frac{M_{j-1}}{L_j}\Big)^{\frac{1}{4}}\bigg] \leq \E\bigg[\Big(\frac{M_{i-1}}{L_i}\Big)^{\frac{1}{2}}\bigg]^{\frac{1}{2}} \E\bigg[\Big(\frac{M_{j-1}}{L_j}\Big)^{\frac{1}{2}}\bigg]^{\frac{1}{2}},
\end{displaymath}
in which the right-hand side is uniformly bounded according to (\ref{eq:1/2-unif-bound}). Moreover, as
\begin{displaymath}
\E\bigg[\Big(M_ic_i+\frac{M_i}{L_{i+1}}\Big)^{\frac{1}{4}} \Big(M_jc_j+\frac{M_j}{L_{j+1}}\Big)^{\frac{1}{4}} \bigg] \leq \E \bigg[\Big((M_ic_i)^{\frac{1}{4}}+\Big(\frac{M_i}{L_{i+1}}\Big)^{\frac{1}{4}}\Big) \Big((M_jc_j)^{\frac{1}{4}}+\Big(\frac{M_j}{L_{j+1}}\Big)^{\frac{1}{4}}\Big) \bigg],
\end{displaymath}
we can develop the right-hand side and similarly use (\ref{eq:C-mean-cv}) and~(\ref{eq:1/2-unif-bound}) to show that it is uniformly bounded. All the other terms can be treated in an analogous way. By similar arguments, one can also prove that $\sup_{i,j}\E[(Q_iQ_j)^2]<\infty$. This finishes the proof of assertion~(iii).
\end{proof}

Using assertions (i) and (ii) of Lemma~\ref{lemma:auxiliary-L1}, we have thus
\begin{equation}
\label{eq:lim-L1}
\lim_{k\to \infty}\E\bigg[\frac{1}{k}\sum\limits_{j=2}^k Q_j\bigg]= \lim_{k\to \infty}\frac{1}{k}\sum\limits_{j=2}^k \E[Q_j] =\E[Q_{\infty}]=\frac{\lambda}{2}.
\end{equation}

\begin{lemma}
\label{lemma:L2-limsup}
We have
\begin{equation}
\label{eq:L2-limsup-Q}
\limsup_{k\to \infty} \E\bigg[\Big(\frac{1}{k}\sum\limits_{j=2}^k Q_j\Big)^2\bigg]\leq \big(\E[Q_{\infty}]\big)^2=\frac{\lambda^2}{4}.
\end{equation}
\end{lemma}

\begin{proof}
We first note that for every $\delta\in(0,1/2)$,
\begin{displaymath}
\Bigg| \E\Big[\sum\limits_{2\leq i,j\leq k}^{} Q_iQ_j\Big]- \E\Big[ \build{\sum_{\delta k\leq i,j\leq k}}_{|i-j|>\delta k}^{} Q_iQ_j\Big] \Bigg|\leq  4\delta k^2 \times \sup_{i,j} \E\big[|Q_iQ_j|\big].
\end{displaymath}
In view of Lemma~\ref{lemma:auxiliary-L1} (iii), the estimate~(\ref{eq:L2-limsup-Q}) will be proved if we can show for arbitrarily small $\delta\in(0,1/2)$ that
\begin{equation}
\label{eq:1st-reduction}
\limsup_{k\to \infty} \frac{1}{k^2}\,\E\bigg[\build{\sum_{\delta k\leq  i,j\leq k}}_{|i-j|>\delta k}^{} Q_iQ_j \bigg]\leq \big(\E[Q_{\infty}]\big)^2.
\end{equation}
We thus fix $\delta\in(0,1/2)$ in the following arguments. By symmetry, we can further restrict our attention to the indices $i$ and $j$ such that $\delta k\leq  i,j\leq k$ and $j-i>\delta k$.

Notice that by the Cauchy--Schwarz inequality,
\begin{align}
\Bigg|\frac{1}{k^2} \,\E\bigg[\build{\sum_{\delta k\leq  i,j\leq k}}_{j-i>\delta k}^{} Q_iQ_j\mathbf{1}_{\{M_{i+1}\geq \varepsilon M_{j-1}\}} \bigg]  \Bigg| &\leq  \frac{1}{k^2} \build{\sum_{\delta k\leq  i,j\leq k}}_{j-i>\delta k}^{}  \P(M_{i+1}\geq \varepsilon M_{j-1})^{\frac{1}{2}} \E\big[(Q_iQ_j)^2\big]^{\frac{1}{2}} \nonumber\\
&\leq  \build{\sup_{\delta k\leq  i,j\leq k}}_{j-i>\delta k}^{}  \P(M_{i+1}\geq \varepsilon M_{j-1})^{\frac{1}{2}} \times \sup_{i,j}\E\big[(Q_iQ_j)^2\big]^{\frac{1}{2}}. \label{eq:QiQj-estimate}
\end{align}
However, observe that Lemma~\ref{lem:kn-asymptotic} can be reformulated as
\begin{displaymath}
\frac{\log M_k}{k} \, \build{\longrightarrow}_{k\to\infty}^{\P-\mathrm{a.s.}}\, \frac{1}{2},
\end{displaymath}
and it follows that for all $\varepsilon\in (0,1)$,

\begin{equation}
\label{eq:bad-proba-0}
\lim_{k\to\infty}\P\big(\mbox{there exist } i,j\in [\delta k,k] \mbox{ with } j-i> \delta k \mbox{ such that } M_{i+1}\geq \varepsilon M_{j-1} \big)=0.
\end{equation}
Together with Lemma~\ref{lemma:auxiliary-L1} (iii), the latter display implies that the right-hand side of (\ref{eq:QiQj-estimate}) converges to 0 as $k\to \infty$. The proof of~(\ref{eq:1st-reduction}) is thus reduced to showing that for fixed $\delta$,
\begin{equation}
\label{eq:2nd-reduction}
\limsup_{\varepsilon \to 0}\Bigg( \limsup_{k\to \infty} \frac{2}{k^2}\,\E\bigg[\build{\sum_{\delta k\leq  i,j\leq k}}_{j-i>\delta k}^{} Q_iQ_j\mathbf{1}_{\{M_{i+1}< \varepsilon M_{j-1}\}} \bigg]\Bigg) \leq \big(\E[Q_{\infty}]\big)^2.
\end{equation}

To this end, we take $\varepsilon\in (0,1/2)$ and define, for every $k\geq 2$,
\begin{displaymath}
h_k^{\varepsilon}= h_k^{\varepsilon}(\widecheck{\mathsf{T}}) \colonequals \mathcal{C}_{M_k-1-\lfloor \varepsilon M_k\rfloor }([\widecheck{\mathsf{T}}]^{*(M_k-1)})
\end{displaymath}
where $[\widecheck{\mathsf{T}}]^{*(M_k-1)}$ stands for the reduced tree associated with $[\widecheck{\mathsf{T}}]^{M_k-1}$ up to height $M_k-1$. In other words, $h_k^{\varepsilon}$ is the probability that a simple random walk on $\widecheck{\mathsf{T}}$ starting from $\mathbf{u}_{M_k-1}$ hits a point of generation $-\lfloor \varepsilon M_k \rfloor$ that has a descendant at generation 0 before hitting $\mathbf{u}_{M_k}$. Comparing with the definition of $h_k$, it is clear that $h_k^{\varepsilon}\geq h_k$. On the other hand, by similar arguments as in the proof of Proposition~\ref{prop:cv-conductance}, we obtain
\begin{displaymath}
h_k^{\varepsilon}-h_k \leq \frac{\lfloor \varepsilon M_k \rfloor}{M_k} h_k^{\varepsilon},
\end{displaymath}
which entails that
\begin{equation}
\label{eq:h-estimation}
M_k h_k^{\varepsilon}-M_kh_k \leq \lfloor \varepsilon M_k \rfloor h_k^{\varepsilon} \leq 2 \varepsilon M_k h_k .
\end{equation}

We set
\begin{displaymath}
Q_k^{\varepsilon}= Q_k^{\varepsilon} (\widecheck{\mathsf{T}}) \colonequals \log \bigg(1+ \frac{M_kc_k+\frac{M_k}{L_{k+1}}}{\frac{M_k}{L_k}}- \frac{\frac{M_k}{L_k}}{\frac{M_k}{L_k}+\frac{M_k}{M_{k-1}}(M_{k-1}c_{k-1}+M_{k-1}h_{k-1}^{\varepsilon})}\bigg) \geq Q_k.
\end{displaymath}
Using the elementary inequality $0\leq \log x-\log y\leq \frac{x-y}{y}$ for $x\geq y>0$, we see that
\begin{align*}
Q_k^{\varepsilon}-Q_k \leq & \, \frac{\frac{M_k}{L_k}}{M_kc_k+\frac{M_k}{L_{k+1}}} \bigg( \frac{\frac{M_k}{L_k}}{\frac{M_k}{L_k}+\frac{M_k}{M_{k-1}}(M_{k-1}c_{k-1}+M_{k-1}h_{k-1})}-\frac{\frac{M_k}{L_k}}{\frac{M_k}{L_k}+\frac{M_k}{M_{k-1}}(M_{k-1}c_{k-1}+M_{k-1}h_{k-1}^{\varepsilon})}\bigg)\\
\leq &\, \frac{\big(\frac{M_k}{L_k}\big)^2}{M_kc_k+\frac{M_k}{L_{k+1}}} \cdot \frac{\frac{M_k}{M_{k-1}}(M_{k-1}h_{k-1}^{\varepsilon}-M_{k-1}h_{k-1})}{\big(\frac{M_k}{L_k}+\frac{M_k}{M_{k-1}}(M_{k-1}c_{k-1}+M_{k-1}h_{k-1}) \big)^2}.
\end{align*}
Taking account of the easy facts that $M_k c_k\geq 1$ and $M_k=M_{k-1}+L_k$, we obtain
\begin{displaymath}
Q_k^{\varepsilon}-Q_k\leq \frac{\big(\frac{M_k}{L_k}\big)^2\frac{M_k}{M_{k-1}}(M_{k-1}h_{k-1}^{\varepsilon}-M_{k-1}h_{k-1})}{\Big(\frac{M_k}{L_k}+\frac{M_k}{M_{k-1}} \Big)^2} \leq M_{k-1}h_{k-1}^{\varepsilon}-M_{k-1}h_{k-1},
\end{displaymath}
which, together with~(\ref{eq:h-estimation}), implies that
\begin{equation*}
Q_k^{\varepsilon}-Q_k \leq  2\varepsilon M_{k-1}h_{k-1}.
\end{equation*}
This allows us to approximate $\E[Q_iQ_j\mathbf{1}_{\{M_{i+1}< \varepsilon M_{j-1}\}}]$ by $\E[Q_iQ_j^{\varepsilon}\mathbf{1}_{\{M_{i+1}< \varepsilon M_{j-1}\}}]$, because
\begin{align*}
\Big|\E\big[Q_iQ_j\mathbf{1}_{\{M_{i+1}< \varepsilon M_{j-1}\}}\big] -\E\big[Q_iQ_j^{\varepsilon}\mathbf{1}_{\{M_{i+1}< \varepsilon M_{j-1}\}}\big] \Big| \,\leq &\, \E\big[|Q_i(Q_j^{\varepsilon}-Q_j)|\big]\\
 \,\leq &\, \E\big[(Q_j^{\varepsilon}-Q_j)^2\big]^{\frac{1}{2}} \times  \E\big[(Q_i)^2\big]^{\frac{1}{2}} \\
 \,\leq &\, 2\varepsilon \,\E\big[(M_{j-1}h_{j-1})^2\big]^{\frac{1}{2}}\times  \E\big[(Q_i)^2\big]^{\frac{1}{2}} ,
\end{align*}
and the right-hand side converges to 0 uniformly with respect to $i,j$ and $k$ when $\varepsilon \to 0$, according to Lemma~\ref{lemma:L2-bdd-hk} and assertion~(iii) of Lemma~\ref{lemma:auxiliary-L1}.

Let us return to the indices $i,j$ such that $\delta k\leq i,j\leq k$ and $j-i>\delta k$, and let $\widetilde{\mathcal{F}}_i$ be the $\sigma$-field generated by the variable $M_{i+1}$ and the finite part of $\widecheck{\mathsf{T}}$ above the vertex $\mathbf{u}_{M_{i+1}}$. Informally, one can think of it as the $\sigma$-field generated by $[\widecheck{\mathsf{T}}]^{M_{i+1}}$. Then $Q_i$ is $\widetilde{\mathcal{F}}_i$-measurable and
\begin{equation}
\label{eq:Q-conditional}
\E\Big[Q_iQ_j^{\varepsilon}\mathbf{1}_{\{M_{i+1}< \varepsilon M_{j-1}\}}\Big]  = \E\Big[Q_i\,\E \Big[Q_j^{\varepsilon}\mathbf{1}_{\{M_{i+1}< \varepsilon M_{j-1}\}}\!\mid\! \widetilde{\mathcal{F}}_i \Big]\Big].
\end{equation}
At this point, we observe that
\begin{equation}
\label{eq:F-tilde-equation}
\E\Big[Q_j^{\varepsilon}\mathbf{1}_{\{M_{i+1}< \varepsilon M_{j-1}\}}\!\mid\! \widetilde{\mathcal{F}}_i \Big]=\E \Big[Q_j^{\varepsilon}\mathbf{1}_{\{M_{i+1}< \varepsilon M_{j-1}\}}\!\mid\! M_{i+1} \Big].
\end{equation}
On the other hand, one can generalize the proof of Lemma~\ref{lemma:5rv-cv-in-law} to show that
\begin{equation}
\label{eq:5tuple-cv-epsilon}
\Big(\frac{L_{j+1}}{M_j},\frac{L_j}{M_{j-1}},M_jc_j,M_{j-1}c_{j-1},M_{j-1}h_{j-1}^{\varepsilon}\Big) \, \build{\longrightarrow}_{j\to\infty}^{(\mathrm{d})}\, \big(\mathcal{R},\mathcal{R}',\mathcal{C},\mathcal{C}', \widehat{\mathcal{C}}_{\varepsilon} \big),
\end{equation}
where in the limit, the first four random variables $\mathcal{R},\mathcal{R}',\mathcal{C}, \mathcal{C}'$ are the same as in Lemma~\ref{lemma:5rv-cv-in-law}, whereas $\widehat{\mathcal{C}}_{\varepsilon}$ is distributed as $\mathcal{C}(\widehat \Delta_{\varepsilon})$. It is assumed in addition that $\mathcal{R},\mathcal{R}',\mathcal{C}, \mathcal{C}', \widehat{\mathcal{C}}_{\varepsilon}$ are independent under $\P$. Furthermore, we can verify that the convergence~(\ref{eq:5tuple-cv-epsilon}) is still valid if, instead of the distribution of
$$\Big(\frac{L_{j+1}}{M_j},\frac{L_j}{M_{j-1}},M_jc_j,M_{j-1}c_{j-1},M_{j-1}h_{j-1}^{\varepsilon}\Big),$$
we consider the conditional distribution of the same random $5$-tuple given $M_{i+1}$, and let $i$ and~$j$ tend to infinity satisfying that $j-i>\delta j$.

We define thus
\begin{displaymath}
Q_{\infty}^{\varepsilon} \colonequals \log \bigg( 1+ \frac{\mathcal{C}+\frac{1}{\mathcal{R}}}{\frac{1+\mathcal{R}'}{\mathcal{R}'}}- \frac{1}{1+\mathcal{R}'\big(\mathcal{C}'+\widehat{\mathcal{C}}_{\varepsilon}\,\big)} \bigg),
\end{displaymath}
which has the limiting distribution of the sequence $(Q_j^{\varepsilon})$ conditioned on $M_{i+1}$. By similar arguments used in the proof of assertion (i) of Lemma~\ref{lemma:auxiliary-L1}, we know that there exists some $p>1$ such that 
\begin{displaymath}
\sup_{\ell\colon \P(M_{i+1}=\ell)>0} \E\big[(Q_j^{\varepsilon})^p \!\mid\! M_{i+1}=\ell\,\big]
\end{displaymath}
is uniformly bounded for all $i,j$ satisfying that $j-i>\delta j$, and that 
\begin{displaymath}
\build{\lim_{i,j\to \infty}}_{j-i>\delta j}^{} \bigg(\sup_{\ell\colon \P(M_{i+1}=\ell)>0} \Big|\,\E\big[\,Q_j^{\varepsilon} \!\mid\! M_{i+1}=\ell\,\big]- \E\big[Q_{\infty}^{\varepsilon}\big]\Big|\bigg)=0.
\end{displaymath}
In view of~(\ref{eq:bad-proba-0}) and (\ref{eq:F-tilde-equation}), it follows that a.s.
\begin{displaymath}
\build{\lim_{i,j\to \infty}}_{j-i>\delta j}^{} \Big| \E\big[Q_j^{\varepsilon}\mathbf{1}_{\{M_{i+1}< \varepsilon M_{j-1}\}}\!\mid\! \widetilde{\mathcal{F}}_i \big] - \E\big[Q_{\infty}^{\varepsilon}\big]\Big|=0.
\end{displaymath}
Hence, we get from (\ref{eq:Q-conditional}) and Lemma~\ref{lemma:auxiliary-L1} (i) that
\begin{equation*}
\lim_{k\to \infty} \Bigg( \build{\sup_{i,j\in [\delta k,k]}}_{j-i>\delta k}^{} \E\big[Q_iQ_j^{\varepsilon}\mathbf{1}_{\{M_{i+1}< \varepsilon M_{j-1}\}}\big] \Bigg)\leq \E[Q_{\infty}]\E\big[Q_{\infty}^{\varepsilon}\big].
\end{equation*}

Finally, it remains to estimate the difference between $\E\big[Q_{\infty}^{\varepsilon}\big]$ and $\E[Q_{\infty}]$. To do this, we use a coupling argument by defining both $\widehat{\mathcal{C}}_{\varepsilon}$ and $\widehat{\mathcal{C}}$ from a common reduced tree $\widehat \Delta$, independent of $(\mathcal{R},\mathcal{R}',\mathcal{C}, \mathcal{C}')$, so that $\widehat{\mathcal{C}} =\mathcal{C}(\widehat \Delta)$ and $\widehat{\mathcal{C}}_{\varepsilon} =\mathcal{C}(\widehat \Delta_{\varepsilon})$. Since $\mathcal{C}\geq 1$ and $\widehat{\mathcal{C}}_{\varepsilon}\geq \widehat{\mathcal{C}}\geq 1$, one can proceed in the same way as we did for bounding $Q_k^{\varepsilon}-Q_k$, and arrive at
\begin{displaymath}
0\leq Q_{\infty}^{\varepsilon}-Q_{\infty} \leq \frac{1+\frac{1}{\mathcal{R}'}}{\mathcal{C}+\frac{1}{\mathcal{R}}} \frac{\mathcal{R}'(\widehat{\mathcal{C}}_{\varepsilon}-\widehat{\mathcal{C}})}{(1+\mathcal{R}'(\mathcal{C}'+\widehat{\mathcal{C}}))^2} \leq  \Big(1+\frac{1}{\mathcal{R}'}\Big) \frac{\mathcal{R}'(\widehat{\mathcal{C}}_{\varepsilon}-\widehat{\mathcal{C}})}{(1+\mathcal{R}')^2} \leq  \widehat{\mathcal{C}}_{\varepsilon}-\widehat{\mathcal{C}}.
\end{displaymath}
Taking account of~(\ref{eq:cond-continu-epsilon}), the last display gives $\big|\E\big[Q_{\infty}^{\varepsilon}\big]-\E[Q_{\infty}]\big|\leq 2\varepsilon\,\E[\widehat{\mathcal{C}}]$, and the right-hand side converges to 0 as $\varepsilon \to 0$.

According to the previous discussions, we conclude that
\begin{align*}
&\; \limsup_{\varepsilon \to 0}\Bigg( \limsup_{k\to \infty} \frac{2}{k^2}\E\bigg[\build{\sum_{\delta k\leq  i,j\leq k}}_{j-i>\delta k}^{} Q_iQ_j\mathbf{1}_{\{M_{i+1}< \varepsilon M_{j-1}\}} \bigg]\Bigg) \\
= &\; \limsup_{\varepsilon \to 0}\Bigg( \limsup_{k\to \infty} \frac{2}{k^2}\E\bigg[\build{\sum_{\delta k\leq  i,j\leq k}}_{j-i>\delta k}^{} Q_iQ_j^{\varepsilon}\mathbf{1}_{\{M_{i+1}< \varepsilon M_{j-1}\}} \bigg]\Bigg) \\
\leq &\;  \limsup_{\varepsilon \to 0} \E[Q_{\infty}] \E\big[Q^{\varepsilon}_{\infty}\big] \,=\, \big(\E[Q_{\infty}]\big)^2,
\end{align*}
which finishes the proof of~(\ref{eq:2nd-reduction}). The proof of Lemma~\ref{lemma:L2-limsup} is therefore completed.
\end{proof}

\smallskip
\noindent{\it Proof of Lemma~\ref{lem:sum-L2-conv}.} By combining~(\ref{eq:lim-L1}) and (\ref{eq:L2-limsup-Q}), we have
\begin{displaymath}
\limsup_{k\to\infty} \E\bigg[\Big(\frac{1}{k}\sum\limits_{j=2}^k Q_j -\frac{\lambda}{2} \Big)^2\bigg] \leq \bigg( \limsup_{k\to\infty} \E\bigg[\Big(\frac{1}{k}\sum\limits_{j=2}^k Q_j\Big)^2\bigg] \bigg) -\lambda \lim_{k\to\infty}\E\bigg[\frac{1}{k}\sum\limits_{j=2}^k Q_j\bigg]+ \frac{\lambda^2}{4} \leq 0,
\end{displaymath}
which gives the desired result.

\end{document}